\documentclass[runningheads]{llncs}
\usepackage{graphicx}
\usepackage{lmodern}
\usepackage{longtable}
\usepackage[normalem]{ulem}
\usepackage{wrapfig}
\usepackage{ucs}
\usepackage{dsfont}
\usepackage{mathrsfs}
\usepackage{mathtools}
\usepackage{amsmath}
\usepackage{amsfonts}
\usepackage{amssymb}
\usepackage{soul}
\usepackage[makeroom]{cancel}
\usepackage[bb=boondox]{mathalfa}
\usepackage[english]{babel}
\usepackage{ucs}
\usepackage{xcolor}
\usepackage[colorlinks=true,linkcolor=blue,citecolor=blue,urlcolor=orange]{hyperref}
\usepackage{longtable}
\usepackage{wasysym}
\usepackage{verbatim}
\usepackage{multirow}
\usepackage{bussproofs}
\usepackage{setspace}
\usepackage{graphicx}
\usepackage{stmaryrd}
\usepackage{pdflscape}
\usepackage{rotating}
\usepackage{forest}
\forestset{smullyan tableaux/.style={for tree={math content},where n children=1{!1.before computing xy={l=\baselineskip},!1.no edge}{},closed/.style={label=below:$\times$},},}
\usepackage{longtable}
\usepackage[all]{xy}
\usepackage{wrapfig}
\usepackage[labelformat=original]{caption}
\usepackage[dvipsnames]{xcolor}
\usepackage[colorinlistoftodos,prependcaption,textsize=small]{todonotes}
\usepackage[most]{tcolorbox}
\usepackage{enumitem}
\usepackage{xfrac}

\newcommand{\KinvG}{\mathbf{K}\mathsf{G}_{\mathsf{inv}}}

\newcommand{\sfive}{\mathbf{S5}}
\newcommand{\crispsfiveinvG}{\mathbf{S5}\mathsf{G}^\mathsf{c}_{\mathsf{inv}}}
\newcommand{\coimplies}{\Yleft}

\newcommand{\KG}{\mathbf{K}\mathsf{G}}
\newcommand{\crispKG}{\mathbf{K}\mathsf{G}^\mathsf{c}}

\newcommand{\pspace}{\mathsf{PSpace}}
\newcommand{\np}{\mathsf{NP}}
\newcommand{\conp}{\mathsf{coNP}}
\newcommand{\Prop}{\mathtt{Prop}}

\newcommand{\invol}{{\sim}}

\newcommand{\modalLinv}{\mathscr{L}^\invol_\Box}

\newcommand{\ClusterSet}{\mathfrak{Cl}}
\newcommand{\cluster}{\mathfrak{cl}}
\newcommand{\cl}{\mathsf{cl}}
\newcommand{\eF}{\mathsf{eF}}

\newcommand{\Amsf}{\mathsf{A}}

\newcommand{\Fmsf}{\mathsf{F}}
\newcommand{\Gmsf}{\mathsf{G}}

\newcommand{\Rmsf}{\mathsf{R}}

\newcommand{\Tmsf}{\mathsf{T}}

\newcommand{\cmsf}{\mathsf{c}}

\newcommand{\Amc}{\mathcal{A}}
\newcommand{\Bmc}{\mathcal{B}}

\newcommand{\Dmc}{\mathcal{D}}

\newcommand{\Hmc}{\mathcal{H}}

\newcommand{\Nmc}{{\mathcal{N}}}
\newcommand{\Omc}{{\mathcal{O}}}
\newcommand{\Pmc}{{\mathcal{P}}}

\newcommand{\Tmc}{{\mathcal{T}}}

\newcommand{\Zmc}{{\mathcal{Z}}}

\newcommand{\cmc}{\mathcal{c}}

\newcommand{\lmc}{\mathcal{l}}

\newcommand{\pmc}{\mathcal{p}}
\newcommand{\tmc}{\mathcal{t}}

\newcommand{\Ffrak}{\mathfrak{F}}
\newcommand{\Gfrak}{\mathfrak{G}}

\newcommand{\Mfrak}{{\mathfrak{M}}}

\newcommand{\Tfrak}{{\mathfrak{T}}}

\newcommand{\nmbf}{\mathbf{n}}
\newcommand{\tmbf}{\mathbf{t}}

\newcommand{\nmbb}{\mathbb{n}}

\newcommand{\Nmbb}{\mathbb{N}}

\newcommand{\TcrispsfiveinvG}{\mathcal{T}\!(\crispsfiveinvG)}
\newcommand{\real}{\mathsf{rl}}
\newcommand{\Var}{\mathsf{Var}}

\newcommand{\Str}{\mathsf{Str}}

\newcommand{\valueterm}{\mathsf{VT}}
\newcommand{\LF}{\mathsf{LF}}
\newcommand{\relterm}{\mathsf{RT}}
\newcommand{\zero}{\mathbb{0}}
\newcommand{\one}{\mathbb{1}}
\newcommand{\onetop}{\mathbf{1}}
\newcommand{\zerobot}{\mathbf{0}}
\newcommand{\WorldLabels}{\mathsf{WL}}
\newtheorem{convention}{Convention}
\usepackage{comment}
\usepackage{boondox-cal}
\usepackage{rotating}
\usepackage{thm-restate}
\usepackage{tikz}
\usepackage{circuitikz}
\begin{document}
\allowdisplaybreaks
\setlength{\jot}{0pt} 
\setlength{\abovedisplayskip}{2pt}
\setlength{\belowdisplayskip}{2pt}
\setlength{\abovedisplayshortskip}{1pt}
\setlength{\belowdisplayshortskip}{1pt}
\setlength{\abovecaptionskip}{5pt plus 3pt minus 2pt}
\setlength{\belowcaptionskip}{5pt plus 3pt minus 2pt}
\title{Tableaux for epistemic G\"{o}del logic
\thanks{The research of Marta B\'ilkov\'a was supported by the grant 22-23022L CELIA of Grantov\'{a} Agentura \v{C}esk\'{e} Republiky.
}
}
%
\author{Marta B\'ilkov\'a\inst{1}\orcidID{0000-0002-3490-2083} \and Thomas Ferguson\inst{1,2}\orcidID{0000-0002-6494-1833}\and Daniil Kozhemiachenko\inst{3}\orcidID{0000-0002-1533-8034}}
\authorrunning{B\'ilkov\'a et al.}
\institute{The Czech Academy of Sciences, Institute of Computer Science, Prague\\
\email{bilkova@cs.cas.cz}
\and
Department of Cognitive Science, Rensselaer Polytechnic Institute, Troy, USA\\
\email{tferguson@gradcenter.cuny.edu}
\and
Aix-Marseille Univ, CNRS, LIS, Marseille, France\\
\email{daniil.kozhemiachenko@lis-lab.fr}}
\maketitle              
\begin{abstract}
We propose a multi-agent epistemic logic capturing reasoning with degrees of plausibility that agents can assign to a~given statement with $1$ interpreted as ‘entirely plausible for the agent’ and $0$ as ‘completely implausible’ (i.e., the agent knows that the statement is false). We formalise such reasoning in an expansion of G\"{o}del fuzzy logic with an involutive negation and multiple $\sfive$-like modalities. As already G\"{o}del single-modal logics are known to lack the finite model property w.r.t.\ their standard $[0,1]$-valued Kripke semantics, we provide an alternative semantics that allows for  the finite model property. For this semantics, we construct a~strongly terminating tableaux calculus that allows us to produce finite counter-models of non-valid formulas. We then use the tableaux to show that the validity problem in our logic is $\pspace$-complete when there are two or more agents, and $\conp$-complete for the single-agent case.

\keywords{epistemic logic \and constraint tableaux \and G\"{o}del logic \and involutive negation \and modal logic.}
\end{abstract}
\section{Introduction\label{sec:introduction}}
Epistemic modal logic deals with reasoning about knowledge and thus usually interprets formulas of the form $\Box_a\phi$ as ‘$a$ knows that $\phi$ is true’. The dual formula $\lozenge_a\phi$ can be then interpreted as ‘$\phi$ is consistent with $a$'s knowledge’ or ‘$a$~considers $\phi$ plausible’. In the classical epistemic logic, such statements are either true or false; however, in many contexts, it makes sense to assume that a~proposition can be \emph{true to a~degree}. E.g., a~person can be ‘very tall’, ‘quite tall’, ‘not so tall’, etc. Similarly, an agent can consider a~given statement $\phi$ \emph{plausible to some degree} --- from \emph{entirely plausible} to \emph{completely implausible}. In the latter case, we can say that a~rational agent \emph{knows that $\phi$ is false}.

There are several approaches to representing degrees of belief and plausibility. The first one is to use \emph{graded modal logics} as has been done in, e.g.,~\cite{vanderHoekMeyer1992}. In this approach, the language is expanded with modal formulas $\lozenge^{\geq n}\phi$ read as ‘$\phi$ is true in at least~$n$ accessible states’. A~similar approach proposed in~\cite{LoriniSchwarzentruber2021} is to add formulas $\Box^n_a\phi$ interpreted as ‘$a$ believes in~$\phi$ after any removal of at most $n$ pieces of information from their belief base’. Another option~\cite{Gardenfors1975,vanderHoek1996} is to consider a~\emph{binary modality} $\succeq$ s.t.\ $\phi\succeq\chi$ is interpreted as ‘$\phi$ is at least as likely / plausible / probable / \ldots as $\chi$’. Note, however, that these logics usually expand \emph{classical} propositional logic. Thus, even though a~belief statement has degrees, it is still not trivial to formalise statements such as ‘the degree to which $a$~finds it plausible that Paula is \emph{very tall} should be lower or equal to the degree to which $a$~finds it possible that she is \emph{quite tall}’ because ‘very tall’ and ‘quite tall’ do not correspond to classical values of formulas.

Thus, in this paper, we will consider an alternative approach that relies on \emph{non-classical}, namely, \emph{fuzzy} logics. In fuzzy modal logics, both propositional and modal formulas can have values in the $[0,1]$ real-valued interval. Moreover, we will concentrate on a~less formal notion of ‘plausibility’ that one may use in everyday reasoning.

In the context of everyday reasoning, people usually do not assign specific values to statements about plausibility but can \emph{compare} them --- e.g., a~hailstorm can be less plausible than rain. To formalise such contexts, one can use \emph{G\"{o}del logic} and its modal expansions. An important property of G\"{o}del logic is that the truth of a~formula depends on the \emph{order} of the values, not on the values themselves. This is because (propositional) G\"{o}del logic is given via the t-norm~$\wedge_\Gmsf$ (G\"{o}del conjunction) and its residuum $\rightarrow_\Gmsf$ (implication):
\begin{align*}
a\wedge_\Gmsf b&=\min(a,b)&
a\rightarrow_\Gmsf b&=\begin{cases}1&\text{ if }a\leq b\\b&\text{ otherwise}\end{cases}
\end{align*}
\textbf{G\"{o}del modal and description logics}
G\"{o}del modal logics are usually defined on Kripke frames of the form $\langle W,R\rangle$ where $W$ is a~set of states and $R$ is a~relation on $W$ (in which case, the frame is called \emph{crisp}) or a~function $W\times W\rightarrow[0,1]$ (in which case, the frame is \emph{fuzzy}).\footnote{Note, however, that some temporal G\"{o}del logics~\cite{AguileraDieguezFernandez-DuqueMcLean2022,AguileraDieguezFernandez-DuqueMcLean2022KR} are constructed over \emph{bi-relational} frames where one relation is used to define modalities and the other to compute the value of implication.} The values of modal formulas $\Box\phi$ and $\lozenge\phi$ in a~given state are defined via, respectively, the infimum and supremum of the values of $\phi$ in accessible states. Modal expansions of G\"{o}del logics are well-studied. In particular, $\Box$ and $\lozenge$ fragments of G\"{o}del modal logic $\KG$ were axiomatised in~\cite{CaicedoRodriguez2010} and $\KG$ with both $\Box$ and $\lozenge$ was axiomatised in~\cite{CaicedoRodriguez2015,RodriguezVidal2021}. Hypersequent calculi were constructed in~\cite{MetcalfeOlivetti2009,MetcalfeOlivetti2011} and used to obtain the $\pspace$-completeness of both fragments. Decidability and $\pspace$-completeness of the full logic were shown in~\cite{CaicedoMetcalfeRodriguezRogger2013,CaicedoMetcalfeRodriguezRogger2017}.

In \emph{knowledge representation and reasoning}, modal logics correspond to \emph{description logics} (DLs). G\"{o}del DLs were proposed in~\cite{BobilloDelgadoGomez-RamiroStraccia2009} to represent graded information in the ontologies and were further studied in~\cite{BobilloDelgadoGomez-RamiroStraccia2012}. It was shown that even the most expressive G\"{o}del DLs often have the same complexity as their classical counterparts~\cite{BorgwardtDistelPenaloza2014DL,BorgwardtDistelPenaloza2014KR,Borgwardt2014PhD,BorgwardtPenaloza2017}. Note, however, that, in contrast to modal logics, the interpretations in fuzzy description logics are often assumed to be \emph{witnessed}, that is, if the value of a~quantified concept assertion $[\forall\Rmsf.\Amsf](a)$ is $x$, then there must be some individual $b$ s.t.\ $\Amsf(b)$ has value~$x$.%

\noindent
\textbf{Epistemic and doxastic G\"{o}del logics}
In classical logic, it is customary (cf., e.g.,~\cite{HalpernMoses1992,FaginHalpernMosesVardi2003,vanDitmarschvanderHoekKooi2007,BlackburndeRijkeVenema2010}) to use $\sfive$ (the logic of frames $\langle W,R\rangle$ where $R$ is an equivalence relation) to formalise reasoning about knowledge and $\mathbf{K45}$ and $\mathbf{KD45}$ (logics of transitive Euclidean and serial transitive Euclidean frames, respectively) to reason about beliefs. Applications of G\"{o}del modal logics to the reasoning about beliefs and knowledge have also received much attention. In particular, G\"{o}del counterparts of logics $\sfive$~\cite{CaicedoMetcalfeRodriguezRogger2013,CaicedoRodriguez2015,Rogger2016phd,CaicedoMetcalfeRodriguezRogger2017,RodriguezVidal2021}, $\mathbf{K45}$, and $\mathbf{KD45}$~\cite{RodriguezTuytEstevaGodo2022} were axiomatised  and shown to be decidable. In fact, the single-agent $\sfive$ turns out to be $\np$-complete just as the classical version. Moreover, a~G\"{o}del analogue of the public announcement logic was proposed in~\cite{BenevidesMadeiraMartins2022}.

\noindent
\textbf{Contributions and plan of the paper}
Even though epistemic and doxastic G\"{o}del logics are well investigated, to the best of our knowledge, the decidability results concern \emph{single-agent} logics. This is why, in this paper, we will study $\crispsfiveinvG$~--- a~G\"{o}del counterpart to $\sfive_n$ (classical $\sfive$ with $n$ agents) over \emph{crisp} frames and expanded with an involutive negation $\invol$ defined as $v(\invol\phi,w)=1-v(\phi,w)$. This negation not only allows one to define additional connectives but also makes $\lozenge$ and $\Box$ interdefinable in the expected fashion.

Moreover, the involutive negation is closer to the intuitive reading of ‘not’ in contexts involving truth degrees than the standard G\"{o}delian negation $\neg_\Gmsf\phi\coloneq\phi\rightarrow_\Gmsf0$. Consider the following example, borrowed from~\cite{BilkovaFergusonKozhemiachenko2025TARK}. Let the truth degree of \emph{‘Paula is tall’} ($p$) be $\sfrac{1}{2}$. It is reasonable to assume that the truth degree of \emph{‘Paula is not tall’} is also $\sfrac{1}{2}$. Similarly, if she is shorter than average (i.e., the truth degree of $p$ is smaller than $\sfrac{1}{2}$), it is reasonable that the truth degree of ‘not-$p$’ should be greater than $\sfrac{1}{2}$. On the other hand, the truth degree of $\neg_\Gmsf p$ is $0$ when the truth degree of $p$ is positive. Thus, using $\neg_\Gmsf p$ to stand for ‘Paula is not tall’ is counterintuitive as it says that ‘Paula is tall is contradictory’.

As G\"{o}del counterpart of $\sfive$ lacks finite model property (FMP) w.r.t.\ standard semantics, we will adapt the approach of~\cite{CaicedoMetcalfeRodriguezRogger2013} and~\cite{BilkovaFergusonKozhemiachenko2025TARK} to produce a~semantics over finitary models (so-called \emph{$\Fmsf$-models}) w.r.t.\ which the logic will have the FMP. We will then use these semantics to construct a~tableaux calculus for $\crispsfiveinvG$ that allows for an explicit construction of countermodels from complete open branches. We will also use this calculus to prove $\pspace$-completeness of $\crispsfiveinvG$ and show that the alternative semantics is equivalent to the standard one.

The remainder of the text is structured as follows. We present and discuss two semantics of epistemic G\"{o}del logic in Section~\ref{sec:Fmodels}: first, the standard semantics, then the one based on $\Fmsf$-models. In Section~\ref{sec:tableaux}, we construct a~tableaux calculus and show its soundness and completeness w.r.t.\ $\Fmsf$-models. Section~\ref{sec:complexity} is dedicated to the complexity of the validity of $\crispsfiveinvG$ w.r.t.\ $\Fmsf$-models. In particular, we establish the finite model property and show that the validity w.r.t.\ $\Fmsf$-models in \emph{multi-agent} $\crispsfiveinvG$ is $\pspace$-complete and $\np$-complete for \emph{single-agent} $\crispsfiveinvG$. In Section~\ref{sec:equivalence}, we show that the new semantics is equivalent to the standard one. Finally, we summarise our results and set goals for future work in Section~\ref{sec:conclusion}. Omitted proofs can be found in the appendix.
\section{Two semantics of epistemic G\"{o}del logic\label{sec:Fmodels}}
Let us now present the language and semantics of $\crispsfiveinvG$. We fix a~countable set $\Prop$ of propositional variables. For a~finite set~$\Amsf$ of \emph{agents}, we define the language $\modalLinv(\Amsf)$ via the following grammar with $p\in\Prop$ and $a\in\Amsf$:
\begin{align*}
\phi&\Coloneqq p\mid\invol\phi\mid(\phi\wedge\phi)\mid(\phi\rightarrow\phi)\mid\Box_a\phi
\end{align*}
From now on, we will assume a~fixed $\Amsf$ and mostly omit explicit reference to it.

\begin{definition}[Epistemic frames]\label{def:epistemicframes}
An \emph{epistemic frame for~$\Amsf$} is a~tuple $\Ffrak\!=\!\langle W,\langle R_a\rangle_{a\in\Amsf}\rangle$ s.t.\ $W\!\neq\!\varnothing$ and $R_a$'s are equivalence relations on~$W$.
\end{definition}
\subsection{Standard semantics}
We begin with the standard semantics of $\crispsfiveinvG$ that generalises the Kripke semantics of $\sfive$~\cite{BlackburndeRijkeVenema2010}.
\begin{definition}[$\crispsfiveinvG$-models]\label{def:epistemicmodels}
An \emph{$\crispsfiveinvG$-model for~$\Amsf$} is a~tuple $\Mfrak=\langle\Ffrak,v\rangle$ s.t.\ $\Ffrak$ is an epistemic frame for~$\Amsf$ and $v:\Prop\times W\rightarrow[0,1]$ (\emph{valuation}).
\end{definition}
The valuation is extended to the complex formulas as follows:
\begin{align*}
v(\invol\phi,w)&=1-v(\phi,w)&v(\phi\wedge\chi,w)&=\min(v(\phi,w),v(\chi,w))\\
v(\phi\rightarrow\chi,w)=&\begin{cases}1\text{, if }v(\phi,w)\leq v(\chi,w)\\v(\chi,w)\text{, else}\end{cases}&
v(\Box_a\phi,w)&=\inf\{v(\phi,w')\mid wR_aw'\}
\end{align*}
We say that a~formula $\phi$ is \emph{$\crispsfiveinvG$-satisfiable} if there is a~$\crispsfiveinvG$-model $\Mfrak$ and $w\in\Mfrak$ s.t.\ $v(\phi,w)=1$; $\phi$ is \emph{$\crispsfiveinvG$-valid} if $v(\phi,w)=1$ for every $\Mfrak$ and $w\in\Mfrak$.
\begin{convention}\label{conv:connectives}
Given a formula $\phi$, we use $\lmc(\phi)$ to denote the number of occurrences of symbols in~$\phi$. We will write $\phi\leftrightarrow\chi$ as a shorthand for $(\phi\rightarrow\chi)\wedge(\chi\rightarrow\phi)$ and use the following defined connectives:\footnote{Here, $\coimplies$ is coimplication ($\phi\coimplies\chi$ is interpreted as ‘$\phi$ excludes $\chi$’; cf.~\cite{Rauszer1974,Gore2000} for a~detailed study and~\cite{Wansing2008} for the interpretation), and $\triangle$ is ‘Baaz Delta operator’ (cf.~\cite{Baaz1996} for a~discussion in the context of fuzzy logics).}
\begin{align*}
\onetop&\coloneqq p\rightarrow p&\zerobot&\coloneqq\invol\onetop&\neg\phi&\coloneqq\phi\rightarrow\zerobot\nonumber\\
\phi\vee\chi&\coloneqq\invol(\invol\phi\wedge\invol\chi)&\phi\coimplies\chi&\coloneqq\invol(\invol\chi\rightarrow\invol\phi)&\triangle\phi&\coloneqq\onetop\coimplies(\onetop\coimplies\phi)\\
\lozenge_a\phi&\coloneqq\invol\Box_a\invol\phi
\end{align*}

Furthermore, in $\Box_a\phi$ or $\lozenge_a\phi$, we call the lower index $a\in\Amsf$ the \emph{agent label}.
\end{convention}

Using Definition~\ref{def:epistemicmodels}, we obtain the following semantics of $\coimplies$, $\neg$, $\leftrightarrow$, $\triangle$, and~$\lozenge$:
\begin{align}
\label{equ:connectives}
v(\phi\!\coimplies\!\chi,w)&=
\begin{cases}
v(\phi,w)\text{, if }v(\phi,w)\!>\!v(\chi,w)\\
0\text{, else}
\end{cases}
&v(\neg\phi,w)&=\begin{cases}1\text{, if }v(\phi,w)\!=\!0\\0\text{, else}\end{cases}\nonumber\\
v(\phi\!\leftrightarrow\!\chi,w)&=
\begin{cases}
1\text{, if }v(\phi,w)\!=\!v(\chi,w)\\
\min(v(\phi,w),v(\chi,w))\text{, else}
\end{cases}&v(\triangle\phi,w)&=
\begin{cases}
1\text{, if }v(\phi,w)\!=\!1\\
0\text{, else}
\end{cases}\nonumber\\
v(\phi\!\vee\!\chi,w)&=\max(v(\phi,w),v(\chi,w))&v(\lozenge_a\phi,w)&=\sup\limits_{wR_aw'}\!\!v(\phi,w')
\end{align}
We also observe that non-validity and satisfiability are reducible to one another. Namely, $\phi$ is valid iff $\invol\triangle\phi$ is \emph{unsatisfiable}; $\phi$ is satisfiable iff $\invol\triangle\phi$ is \emph{not valid}.

In the introduction, we mentioned that it makes sense to consider degrees of plausibility in epistemic contexts. Furthermore, agents might know that a~proposition is true \emph{to some degree}. The example below illustrates one of such contexts and highlights the differences between G\"{o}del and classical epistemic logics.
\begin{example}\label{example:formalisation}
Consider the following formula: $\phi_\mathsf{plaus}\coloneqq(\Box_a(\invol p\rightarrow p)\wedge\lozenge_a p)\rightarrow\lozenge_bp$. Here, $\phi_\mathsf{plaus}$ says that if $a$~knows that the value of $p$ is at least $\frac{1}{2}$ and considers~$p$ plausible (does not know that it is false), then $b$ should also consider~$p$ plausible. Observe that $\phi_\mathsf{plaus}$ is \emph{valid} in $\sfive_2$ since $p$ and $\invol p\rightarrow p$ are classically equivalent and $\Box_ap$ implies $\lozenge_bp$. On the other hand, it can be easily refuted in $\crispsfiveinvG$, e.g., in the following model:
\begin{align*}
\Mfrak:\quad\xymatrix{w_1:p=\tfrac{1}{2}~\ar@{-}|{R_b}[rr]&&~w_0:p=\tfrac{1}{2}~\ar@{-}|{R_a}[rr]&&~w_2:p=1}
\end{align*}
Let us now present a~natural context illustrating $\Mfrak$. Say, there are three people: Alice, Brittney, and Paul. Alice \emph{knows} that Paul is \emph{quite} tall (i.e., that the value of $p$ is at least $\tfrac{1}{2}$) and finds it plausible that he is \emph{very} tall (value of $p$ equals~$1$ in some state accessible to Alice). Of course, it does not mean that Brittney also finds it plausible that Paul is \emph{very} tall. On the other hand, it is easy to see that the following formula \emph{is $\crispsfiveinvG$-valid}: $\phi'_\mathsf{plaus}\coloneqq(\Box_a(\invol p\rightarrow p)\wedge\lozenge_a p)\rightarrow(\invol\lozenge_b p\rightarrow\lozenge_b p)$. Here, if Alice knows that Paul is quite tall and finds it plausible that he is very tall, the degree of plausibility of $p$ for Brittney is at least $\frac{1}{2}$.
\end{example}

Finally, we observe that degrees of belief also depend on the degrees of truth. Let us consider such contexts in more detail.
\begin{example}\label{example:quitetall}
Consider two statements: \emph{the suitcase is very heavy} ($s$) and \emph{the suitcase is quite heavy} ($\invol s\rightarrow s$ --- i.e., the truth degree of $s$ is at least as high as that of its negation). It is reasonable to assume that the degree of plausibility of $s$ for Ann should not be higher (but can be lower) than the degree of plausibility of~$\invol s\rightarrow s$. Indeed, one can use Definition~\ref{def:epistemicmodels} and~\ref{conv:connectives} to see that $\lozenge_ap\rightarrow\lozenge_a(\invol p\rightarrow p)$ is $\crispsfiveinvG$-valid but $\lozenge_a(\invol p\rightarrow p)\rightarrow\lozenge_ap$ is not.

On the other hand, in classical logic, one cannot express this connection between different truth degrees of a~single statement and degrees of plausibility. Indeed, as $s$ and $\invol s\rightarrow s$ are classically equivalent, it follows that $\lozenge_ap\leftrightarrow\lozenge_a(\invol p\rightarrow p)$ is $\sfive$-valid.
\end{example}
\subsection{Epistemic $\Fmsf$-models}
As is well-known~\cite{CaicedoMetcalfeRodriguezRogger2013,Rogger2016phd,CaicedoMetcalfeRodriguezRogger2017}, even the single-agent version of the G\"{o}del modal logic over epistemic frames lacks the finite model property. So, will our semantics with Definition~\ref{def:epistemicmodels}. In~\cite{CaicedoMetcalfeRodriguezRogger2017}, there was given a~procedure to define alternative semantics for \emph{order-based} modal logics, i.e., logics whose connectives can be expressed in the lattice language. Clearly, $\invol$ is not order-based. Still, we can combine the approaches from~\cite{CaicedoMetcalfeRodriguezRogger2013} and~\cite{BilkovaFergusonKozhemiachenko2025TARK} to construct such semantics for $\crispsfiveinvG$.
\begin{definition}[$\eF$-models]\label{def:epistemicFmodels}
Let $\Pmc_{<\omega}([0,1])\!=\!\{X\mid X\!\subseteq\![0,1],|X|<\aleph_0\}$. An \emph{epistemic $\Fmsf$-model for~$\Amsf$} ($\eF$-model) is a~tuple $\Mfrak\!=\!\langle W,\!\langle R_a\rangle_{a\in\Amsf},\!\langle T_a\rangle_{a\in\Amsf},v\rangle$ s.t.\
\begin{itemize}[noitemsep,topsep=1pt]
\item $\langle W,\langle R_a\rangle_{a\in\Amsf}\rangle$ is an epistemic frame;
\item $T_a{:}W\!\rightarrow\!\Pmc_{<\omega}([0,1])$ are s.t.\ (i) $\{0,\frac{1}{2},1\}\subseteq T_a(w)$ for each $w\in W$, (ii) if $x\in T_a(w)$, then $1-x\in T_a(w)$, (iii) if $wR_aw'$, then $T_a(w)=T_a(w')$;
\item $v{:}\Prop\!\times\!W\!\rightarrow\![0,1]$.
\end{itemize}
We call $v$ an~\emph{$\Fmsf$-valuation} and extend it to propositional formulas as in Definition~\ref{def:epistemicmodels} and to modal formulas as follows:
\begin{align*}
v(\Box_a\phi,w)&=\max\{x\in T_a(w)\mid x\leq\inf\{v(\phi,w')\mid wR_aw'\}\}
\end{align*}
We say that a~formula $\phi$ is \emph{satisfiable on epistemic $\Fmsf$-models ($\eF$-satisfiable)} if there is an $\eF$-model $\Mfrak$ and $w\in\Mfrak$ s.t.\ $v(\phi,w)=1$; $\phi$ is \emph{valid on epistemic $\Fmsf$-models ($\eF$-valid)} if $v(\phi,w)=1$ for every $\Mfrak$ and $w\in\Mfrak$.
\end{definition}

Let us now discuss $\Fmsf$-models in more detail. As in~\cite{CaicedoMetcalfeRodriguezRogger2013,CaicedoMetcalfeRodriguezRogger2017}, $T_a$'s assign finite subsets of $[0,1]$ to each state in the model that represent sets of values that formulas $\Box_a\phi$ can have at $w$. In a~sense, the value of $\Box_a\phi$ in an $\Fmsf$-model can be interpreted as an approximation from below of its ‘real’ value in a~standard model. Conditions (i) and (ii) account for the presence of~$\invol$ and result in the following semantics of $\lozenge_a\phi$:
\begin{align*}
v(\invol\Box_a\invol\phi,w)=v(\lozenge_a\phi,w)&=\min\{x\in T_a(w)\mid x\geq\sup\{v(\phi,w')\mid wR_aw'\}\}
\end{align*}
This way, the value of $\lozenge_a\phi$ in an $\Fmsf$-model is an approximation from above of its ‘real’ value in a~standard model. Finally, condition~(iii) ensures the preservation of the following property of $\crispsfiveinvG$-models: if $wR_aw'$, then $v(\Box_a\phi,w)=v(\Box_a\phi,w')$ for every $\phi\in\modalLinv$.

We recall from the introduction that in G\"{o}del description logics, it is customary~\cite{BorgwardtDistelPenaloza2014KR,BorgwardtPenaloza2017} to assume that interpretations of ontologies are witnessed. In modal logic terms, this means that if $v(\Box_a\phi,w)=x$ or $v(\lozenge_a\phi,w)=x$, then there is some accessible $w'$ s.t.\ $v(\phi,w)=x$. Thus, a~natural question is whether epistemic $\Fmsf$-models are just witnessed standard models in disguise. The answer is known to be negative~\cite{CaicedoRodriguez2010,CaicedoMetcalfeRodriguezRogger2013}. For instance, $\phi=\triangle\lozenge_a p\rightarrow\lozenge_a\triangle p$ has value~$1$ in every witnessed model. On the contrary, consider the following $\eF$-model $\Mfrak=\langle W,R_a,T_a,v\rangle$: $W=\{w\}$, $R_a=\langle w,w\rangle$, $T_a(w)=\{0,\tfrac{1}{2},1\}$, $v(p,w)=\frac{3}{4}$. Now observe from Definition~\ref{def:epistemicFmodels} that $v(\lozenge p,w)=1$ but $v(\lozenge_a\triangle p,w)=0$. Thus, $v(\triangle\lozenge_a p\rightarrow\lozenge_a\triangle p,w)=0$.

Another question is whether rejecting witnessed models is reasonable in an epistemic context. In the following paragraph, we argue that it does make sense.

Intuitively, $\phi$ tells that if $a$~considers $p$~absolutely plausible, then $a$~should consider it possible that $p$~is true (has value~$1$). Observe, however, that one might not accept this principle if values of $p$ are considered \emph{large enough} to infer the (absolute) plausibility of~$p$. In the standard semantics, this can be modelled by having infinitely many accessible states where the supremum of the values of~$p$ is~$1$ even though $p$~never has value~$1$. In the $\eF$-models, given $T_a(w)=\{0,x_1,\ldots,x_n,1\}$, one can interpret the values of formulas in the interval $(x_n,1)$ as \emph{large enough}. Note furthermore, that as $\frac{1}{2}\in T_a(w)$, the set of negligibly small values does not trivialise to $\{0,1\}$. Dually, the values in the interval $(0,x_1)$ can be considered \emph{negligibly small}. This would mean that the following formula $\chi=\Box_a\neg\neg q\rightarrow\neg\neg\Box_aq$ can be invalidated in the $\eF$-model from the previous paragraph by setting $v(q,w)=\frac{1}{4}$. One can interpret $\chi$ as follows: if $a$~knows that the value of~$q$ is positive, then $a$~should not consider its negation plausible. Indeed, observe that $v(\neg\neg\Box_aq)=1$ iff $v(\lozenge_a\invol q)<1$. If, however, some value $x>0$ of~$q$ is \emph{negligibly small} for~$a$, they may not accept~$\chi$ if $v(q,w')\leq x$ in some state~$w'$ accessible to~$a$.
\section{Tableaux\label{sec:tableaux}}
Many reasoning techniques are used for G\"{o}del modal and description logics. In particular, there are terminating sequent calculi for $\Box$- and $\lozenge$-fragments of the G\"{o}del modal logic~\cite{MetcalfeOlivetti2009,MetcalfeOlivetti2011} (with the standard negation~$\neg$, not the involutive negation~$\sim$). There are also tableaux for G\"{o}del description logics with witnessed semantics, and it is known that they can be reduced to \emph{classical reasoning in ontologies} (cf.~\cite{BorgwardtPenaloza2017} for both approaches applied to expressive G\"{o}del DLs). In the case of \emph{general} (i.e., not necessarily witnessed) interpretations, automata-based decision procedures are used~\cite{BorgwardtDistelPenaloza2014DL}. In addition, there are semantic tableaux systems based on $\Fmsf$-models. In particular, tableaux for $\KG$ and $\crispKG$ (G\"{o}del modal logics \emph{without} involutive negation over arbitrary fuzzy and crisp frames, respectively) and for $\sfive\Gmsf^\cmsf$ (single-agent epistemic G\"{o}del logic without~$\sim$) are presented in~\cite{Rogger2016phd}. In~\cite{BilkovaFergusonKozhemiachenko2025TARK}, a~tableaux calculus for $\KinvG$ (a~G\"{o}del modal logic with involutive negation over \emph{arbitrary} fuzzy frames) is presented.

In this paper, we are constructing a~tableaux calculus for $\crispsfiveinvG$ based on $\eF$-models. As we have seen in the previous section, $\eF$-models \emph{do not} correspond to witnessed models. Thus, it is unreasonable to expect that there is a~nice reduction of $\crispsfiveinvG$-satisfiability to classical $\sfive$-satisfiability. Furthermore, in contrast to reasoning procedures based on sequent calculi or automata (or on reductions to classical logics), tableaux enable straightforward extraction of finite countermodels from failed proofs. Thus, the finite model property (w.r.t.\ $\eF$-models) will follow immediately from the completeness of tableaux.
We will generalise the tableaux for $\Gmsf\sfive^\cmsf$ from~\cite{Rogger2016phd}. One of the features of \emph{mono-modal} $\sfive$-like logics is that they can be characterised via frames with \emph{universal} relations. Thus, simple terminating tableaux for them can be defined without encoding accessibility relation (cf., e.g.,~\cite{Priest2008FromIftoIs} for the classical $\sfive$). Of course, when dealing with \emph{multi-modal}~$\sfive$, one has to keep track of different relations. This can be done, e.g., with relational terms of the form $wRw'$~\cite{HalpernMoses1992}. In this case, one sometimes has to propagate modal formulas into the generated states, which might complicate the decision procedure. Note, however, that since $R_a$'s in epistemic frames are equivalence relations, we can represent frames as families of clusters~--- equivalence classes under~$R_a$'s.

We can now use terms of the form $w\!\in\!\cluster_a$ interpreted as ‘$w$ belongs to an~$R_a$-cluster $\cluster$’ in our tableaux. Furthermore, to account for the involutive negation, we will adapt ‘constraint tableaux’ from~\cite{Haehnle1999} and follow the presentation of the tableaux calculus from~\cite{BilkovaFergusonKozhemiachenko2025TARK}. We begin by giving formal definitions of the needed notions. We will then explain them in further detail.
\begin{definition}[Structures and constraints]\label{def:constraints}
We fix a~countable set $\WorldLabels=\{w,w',w_0,\ldots\}$ of \emph{state-labels}, $\ClusterSet=\bigcup_{a\in\Amsf}\{\cluster_a,\cluster'_a,\cluster^1_a,\ldots\}$ of \emph{cluster-labels}, and $\Var=\{d,e,d',\ldots\}$ of variables. We also define the set of $\Tmsf$-symbols $\Tmsf=\{t_i(\cluster_a)\!\mid\!\cluster_a\!\in\!\ClusterSet,i\in\Nmbb,a\in\Amsf\}\cup\{t^s_i(\cluster_a)\mid\cluster_a\in\ClusterSet,a\in\Amsf,i\in\Nmbb\}\cup\{\zero,\one\}$, and let $\triangledown\in\{\leqslant,<,\geqslant,>,=\}$. We define sets of \emph{labelled formulas} ($\LF$) and \emph{relational terms} ($\relterm$) as follows:
\begin{align*}
\LF&=\{w{:}\phi\mid w\in\WorldLabels,\phi\in\modalLinv\}&\relterm&=\{w\in\cluster_a\mid w\in\WorldLabels,\cluster_a\in\ClusterSet,a\in\Amsf\}%
\end{align*}

The set $\valueterm$ of \emph{value terms} is defined as follows:
\begin{align*}
\valueterm\ni\Tfrak&\coloneqq c\in\Var\mid\tmbf\in\Tmsf\mid0\mid1\mid1-\Tfrak
\end{align*}

Finally, \emph{constraints} have the form $\Sigma\triangledown\Tfrak$ s.t.\ $\Sigma\in\LF\cup\valueterm$, $\Tfrak\in\valueterm$. We will also set $\Str\coloneqq\LF\cup\valueterm$ and call its elements \emph{structures}.
\end{definition}

In the definition above, a~constraint of the form $w{:}\phi\leqslant\Tfrak$ means that the value of $\phi$ in~$w$ is less or equal to~$\Tfrak$. Note that $\Tfrak$ is never a~formula but a~‘value term’ --- either a~number, a~member of $\Tmsf$, or an expression of the form $1-\Tfrak$. Thus, we never compare the values of formulas to one another. We also need three types of $\Tmsf$-symbols --- constants ($\zero$ and $\one$), $t$'s, and $t^s$'s. Here, $t_i(\cluster_a)$ is a~member of $T_a(w)$ for any $w\in\cluster_a$, and $t^s_i(\cluster_a)$ is the immediate successor of $t_i(\cluster_a)$. This is needed to represent the semantics of $\Box$ (recall from Definition~\ref{def:epistemicFmodels} that the value of $\Box_a\phi$ at $w$ is the \emph{maximal} element of $T_a(w)$ that still does not exceed the infimum of the values of $\phi$).
\begin{definition}[$\TcrispsfiveinvG$ --- tableaux for $\crispsfiveinvG$]\label{def:TS5invG}
A \emph{tableau} is a downward-branching tree whose nodes are constraints and relational terms. Each branch can be extended by one of the rules from Figure~\ref{fig:rules}. 

Let $\Bmc=\{\cmc_1,\ldots,\cmc_n\}$ be a branch with constraints $\cmc_1$, \ldots, $\cmc_n$. Given a~constraint $\cmc$, let $\cmc^\tmc$ be the result of replacing $\Tfrak\in\valueterm$, and $\lambda\in\LF$ with variables $x_\Tfrak$, and $x_\lambda$, respectively, and $\nmbb\in\{\zero,\one\}$ with $n\in\{0,1\}$. Furthermore, for every $\cluster_a\in\ClusterSet$ and $a\in\Amsf$ occurring on $\Bmc$, we set
\begin{align*}
\Tmsf(\cluster_a)&=\{t(\cluster_a)\!\mid t(\cluster_a)\text{ is on }\Bmc\}\!\cup\!\{t^s(\cluster_a)\!\mid t^s(\cluster_a)\text{ is on }\Bmc\}\cup\\
&\quad~\{\nmbb\!\mid\exists\phi\;w{:}\phi\!=\!\nmbb\!\in\!\Bmc\text{ and }w\in\cluster_a\text{ occurs on }\Bmc\}
\end{align*}
and define $\Bmc$ to be \emph{closed} iff the following system of inequalities
\begin{align*}
\Bmc^\tmc&=\{\cmc^\tmc_1,\ldots,\cmc^\tmc_n\}\cup\{x_{t(\cluster_a)}<x_{t^s(\cluster_a)}\mid t(\cluster_a)\text{ and }t^s(\cluster_a)\text{ occur on }\Bmc\}
\end{align*}
\emph{does not have a~solution over $[0,\!1]$} s.t.\ for each $w\!\in\!\WorldLabels$ and $a\!\in\!\Amsf$, it holds that:
\begin{align}\label{equ:closure}
\Tmsf(\cluster_a)\!=\!\{t(\cluster_a),t^s(\cluster_a)\}\Rightarrow&\left[\begin{matrix}x_{t(\cluster_a)}=0~\&~x_{t^s(\cluster_a)}=\frac{1}{2}&\text{or }\\x_{t(\cluster_a)}=\frac{1}{2}~\&~x_{t^s(\cluster_a)}=1\end{matrix}\right]\nonumber\\[.4em]
|\Tmsf(\cluster_a)|\geq3\Rightarrow&
\left[\begin{matrix}
\forall\tmbf\!\in\!\Tmsf(\cluster_a)\;\exists\tmbf'\!\in\!\Tmsf(\cluster_a):x_\tmbf=1-x_{\tmbf'}\text{ and}\\[.4em]
\exists\tmbf_1,\tmbf_2,\tmbf_3:x_{\tmbf_1}=0~\&~x_{\tmbf_2}=\frac{1}{2}~\&~x_{\tmbf_3}=1
\end{matrix}\right]\\[.4em]
&\neg\exists t(\cluster_a),t^s(\cluster_a),t'(\cluster_a):x_{t(\cluster_a)}<x_{t'(\cluster_a)}<x_{t^s(\cluster_a)}\nonumber
\end{align}
A~branch is \emph{open} if it is not closed. A~branch $\Bmc$ is \emph{complete} when for every premise of any rule occurring on~$\Bmc$, its conclusion also occurs in $\Bmc$. The only exceptions are branches containing
\begin{itemize}[noitemsep,topsep=1pt]
\item $w{:}\phi<\Tfrak$ and $w{:}\phi<1$, or
\item $w{:}\phi>\Tfrak$ and $w{:}\phi>0$, or
\item $u:\Box_a\phi=\Tfrak$, $u':\Box_a\phi=\Tfrak'$, $u\in\cluster_a$, and $u'\in\cluster_a$.
\end{itemize}
In the first two cases, the rules are applied to the constraints with $\Tfrak$, not $0$ and~$1$. In the third case, we use $\equiv_\cluster$ and then apply $\Box_=$ to $u:\Box_a\phi=\Tfrak$ only.

Finally, $\phi\in\modalLinv$ \emph{has a~$\TcrispsfiveinvG$ proof} if there is a~tableau beginning with $\{w{:}\phi<1\}\cup\{w\in\cluster_a\mid a\in\Amsf\}$ s.t.\ all its branches are closed.
\end{definition}
\begin{figure}
\centering
\begin{align*}
\invol:\dfrac{w{:}\invol\phi\triangledown\Tfrak}{w{:}\phi\blacktriangledown1\!-\!\Tfrak}
&&
\wedge_\triangleright:\dfrac{w{:}\phi\wedge\chi\triangleright\Tfrak}{\begin{matrix}w{:}\phi\triangleright\Tfrak\\w{:}\chi\triangleright\Tfrak\end{matrix}}
&&
\wedge_\triangleleft:\dfrac{w{:}\phi\wedge\chi\triangleleft\Tfrak}{w{:}\phi\triangleleft\Tfrak\mid w{:}\chi\triangleleft\Tfrak}\\[.4em]
\rightarrow_\triangleright:\dfrac{w{:}\phi\rightarrow\chi\triangleright\Tfrak}{w{:}\chi\triangleright\Tfrak\left|\begin{matrix}\Tfrak\triangleleft1\\w{:}\chi\geqslant d\\w{:}\phi\leqslant d\end{matrix}\right.}
&&
\rightarrow_\leqslant:\dfrac{w{:}\phi\rightarrow\chi\leqslant\Tfrak}{1\leqslant\Tfrak\left|\begin{matrix}w{:}\phi>d\\w{:}\chi\leqslant d\\d\leqslant\Tfrak\end{matrix}\right.}
&&
\rightarrow_<:\dfrac{w{:}\phi\rightarrow\chi<\Tfrak}{\begin{matrix}w{:}\phi\geqslant d\\w{:}\chi<d\\d\leqslant\Tfrak\end{matrix}}
\end{align*}
\begin{align*}
\Box_\triangleright:\dfrac{\begin{matrix}w{:}\Box_a\phi\triangleright\Tfrak\\w\in\cluster_a\end{matrix}}{\begin{matrix}w{:}\Box_a\phi\!=\!\one\\1\triangleright\!\Tfrak\end{matrix}\left|\begin{matrix}w{:}\Box_a\phi\!=\!t(\cluster_a)\\\Tfrak\triangleleft t(\cluster_a)\\w'{:}\phi\!<\!t^s(\cluster_a)\\w'\in\cluster_a\\w'\in\cluster'_{a'}\end{matrix}\right.}
&&
\Box_\leqslant:\dfrac{\begin{matrix}w{:}\Box_a\phi\!\leqslant\!\Tfrak\\w\in\cluster_a\end{matrix}}{1\!\leqslant\!\Tfrak\left|\begin{matrix}t(\cluster_a)\!\leqslant\!\Tfrak\\w'{:}\phi\!<\!t^s(\cluster_a)\\w'\in\cluster_a\\w'\in\cluster'_{a'}\end{matrix}\right.}
&&
\Box_<:\dfrac{\begin{matrix}w{:}\Box_a\phi\!<\!\Tfrak\\w\in\cluster_a\end{matrix}}{\begin{matrix}t(\cluster_a)\!<\!\Tfrak\\w'{:}\phi\!<\!t^s(\cluster_a)\\w'\in\cluster_a\\w'\in\cluster'_{a'}\end{matrix}}\\[.4em]
\Box_=:\dfrac{\begin{matrix}w{:}\Box_a\phi\!=\!\Tfrak\\w\in\cluster_a~u\in\cluster_a\end{matrix}}{u{:}\phi\!\geqslant\!\Tfrak}
&&
\equiv_\cluster:\dfrac{\begin{matrix}u:\Box_a\phi=\Tfrak\\u':\Box_a\phi=\Tfrak'\\u\in\cluster_a~ u'\in\cluster_a\end{matrix}}{\begin{matrix}\Tfrak\leqslant\Tfrak'\\\Tfrak'\leqslant\Tfrak\end{matrix}}
\end{align*}
\caption{Tableaux rules. Vertical bars denote branching; $\blacktriangledown,\triangledown\!\in\!\{\leqslant,<,\geqslant,>\}$, if $\triangledown$ is $\leqslant$, then $\blacktriangledown$ is $\geqslant$ and vice versa (likewise for $<$); $\triangleright\in\{\geqslant,>\}$, $\triangleleft\in\{\leqslant,<\}$; $d$, $\cluster'_a$, $w'$, $t(\cluster_a)$, and $t^s(\cluster_a)$ are fresh on the branch; $a'\in\Amsf$ appears on the branch.}
\label{fig:rules}
\end{figure}

Let us briefly explain tableaux rules. First, we note that every generated state is added to a~cluster w.r.t.\ each relation. This is needed because accessibility relations in $\eF$-models are reflexive. Thus, when an application of a~$\Box_\leqslant$, $\Box_<$, or a~$\Box_\triangleright$ rule produces a~new state $w'$, we add it to the current cluster $\cluster_a$ and then also for every $a'\!\in\!\Amsf$ s.t.\ $a'\neq a$, create \emph{new} clusters $\cluster'_{a'}$ that will contain~$w'$. Moreover, $\Box_a\psi$ formulas are guaranteed to have the same values in all states of a~given $R_a$-cluster because of the $\equiv_\cluster$ rule.

Second, consider the $\Box_\leqslant$ rule. There we consider two possibilities: either $\Box_a\phi$ has value $1$ at~$w$ (left branch) or it has some value $t(\cluster_a)$ from $T_a(w)$ strictly lower than $1$ (right branch). In the first case, $\phi$ will have value $1$ in all states of the $R_a$-cluster to which $w$ belongs (cf.~$\Box_=$). In the second case, we create $w'$ where $\phi$ has value \emph{between $t(\cluster_a)$ and its successor $t^s(\cluster_a)$}. After the application of $\Box_\leqslant$, we can use $\Box_=$ that stipulates that $\phi$ has value at least $t(\cluster_a)$ in all states of the current $R_a$-cluster. Note, moreover, that $\Box_<$ rule is, essentially, a~particular case of the $\Box_\leqslant$ rule. We keep them separate to avoid unnecessary introductions of constraints $1<1$ to the branch.

Third, observe that the clusters allow us to avoid redundant applications of modal rules. E.g., if $\Bmc$~contains $u:\Box_a\phi\leqslant\Tfrak$, $u':\Box_a\phi\leqslant\Tfrak$, $u\!\in\!\cluster_a$, and $u'\!\in\!\cluster_a$, it is clear that we only need to apply the $\Box_\triangleleft$-rule once. This is because the newly generated state $w'$ will belong to the same $R_a$-cluster~$\cluster_a$ as both~$u$ and~$u'$. Similarly, if $\Bmc$ contains $u:\Box_a\phi\geqslant\Tfrak$, $u':\Box_a\phi\geqslant\Tfrak'$, $u\!\in\!\cluster_a$, and $u'\!\in\!\cluster_a$, we use the $\equiv_\cluster$ rule to avoid redundant applications of modal rules.

\begin{definition}[Model realising a~branch]\label{def:realisingmodel}
Let $\Mfrak=\langle W,\langle R_a\rangle_{a\in\Amsf},\langle T_a\rangle_{a\in\Amsf},v\rangle$ be an $\eF$-model and $\Bmc$ a~tableau branch. An~\emph{$\Mfrak$-realisation of~$\Bmc$} is a~map $\real:\WorldLabels\cup\Str\rightarrow W\cup[0,1]$ s.t.\ $\real(w)\in W$ and $\real(\Sigma)\in[0,1]$ for each $w\in\WorldLabels$ and $\Sigma\in\Str$ occurring on~$\Bmc$, and the following properties hold:
\begin{enumerate}[noitemsep,topsep=1pt]
\item $\real(\zero)=0$, $\real(\one)=1$;
\item $\real(w)R_a\real(w')$ if $w\in\cluster_a$ and $w'\in\cluster_a$ occur in~$\Bmc$ for some clus\-ter-label $\cluster_a$;
\item if $\real(\Tfrak)=x$, then $\real(1-\Tfrak)=1-x$ for every $\Tfrak\in\valueterm$;
\item $(\{\real(\tmbf)\mid\tmbf\in\Tmsf(\cluster_a)\}\cup\{0,\frac{1}{2},1\})\subseteq T_a(\real(w))$ and $\real(t(\cluster_a))<\real(t^s(\cluster_a))$ for all $w\in\WorldLabels$ and $a\in\Amsf$ s.t.\ $w\in\cluster_a$ occurs on~$\Bmc$;
\item there are no $t'(\cluster_a)$, $t(\cluster_a)$, $t^s(\cluster_a)$ on~$\Bmc$ s.t.\ \mbox{$\real(t(\cluster_a))\!\!<\!\real(t'(\cluster_a))\!\!<\!\real(t^s(\cluster_a))$;}
\item if $\Tmsf(\cluster_a)=\{t(\cluster_a),t^s(\cluster_a)\}$, then $\frac{1}{2}\in\{\real(t(\cluster_a)),\real(t^s(\cluster_a))\}$;
\item if $|\Tmsf(\cluster_a)|\geq3$, then for each $\tmbf\in\Tmsf(\cluster_a)$ on $\Bmc$, there is $\tmbf'\in\Tmsf(\cluster_a)$ on $\Bmc$ s.t.\ $\real(\tmbf)=1\!-\!\real(\tmbf')$, and there are $t_1,t_2,t_3\!\in\!\Tmsf(\cluster_a)$ s.t.\ $\real(t_1)\!=\!0$, $\real(t_2)\!=\!\frac{1}{2}$, and $\real(t_3)\!=\!1$.
\end{enumerate}
A~constraint $w{:}\phi\triangledown\Tfrak$ is \emph{realised by~$\Mfrak$ under $\real$} if $v(\phi,\real(w))\triangledown\real(\Tfrak)$. A~constraint $\Tfrak\triangledown\Tfrak'$ is realised by~$\Mfrak$ under $\real$ if $\real(\Tfrak)\triangledown\real(\Tfrak')$. A~branch~$\Bmc$ is realised by~$\Mfrak$ under $\real$ if \emph{$\real$ realises all constraints occurring on~$\Bmc$}.
\end{definition}

Let us now give an example of a~failed proof in~$\TcrispsfiveinvG$. We will see how to construct a~realising model from a~complete open branch.
\begin{example}\label{example:failedproof}
Recall the formula $\phi_\mathsf{plaus}$ from Example~\ref{example:formalisation} and replace $\lozenge$'s with $\Box$'s as follows: $(\Box_a(\invol p\!\rightarrow\!p)\!\wedge\!\invol\Box_a\invol p)\!\rightarrow\!\invol\Box_b\invol p$. A~failed proof is given in Fig.~\ref{fig:tableauxexamplepart1}.

First, consider the closed branches. All of them but the ones marked $\times_A$ and $\times_B$ contain constraints whose incompatibility is evident. In the leftmost branch: $w_0{:}p\geqslant\one$ and $w_0{:}p\leqslant1-\one$. In the second branch from the left: $w_0{:}p\leqslant1-\one$, $w_0{:}p\geqslant d_1$, and $w_0{:}p\geqslant1-d_1$. In the third branch: $w_0{:}p\geqslant t(\cluster^0_a)$, $t(\cluster^0_a)\geqslant d_0$, $w_0{:}p\leqslant1-\one$, and $w_0{:}\invol\Box_b\invol p<d_0$; etc. To see why $\times_A$ is closed, observe that $w_3{:}p<d_1$ and $w_3{:}p\leqslant1-d_1$ imply that $v(p,w_3)<\tfrac{1}{2}$. Now as $d_0>0$ (because of $w_0{:}\invol\Box_b\invol p<d_0$), we know that $t_1(\cluster^0_a)>0$ as well. Note that $\Tmsf(\cluster^0_a)=\{t_0(\cluster^0_a),t^s_0(\cluster^0_a),t_1(\cluster^0_a),t^s(\cluster^0_a)\}$. $\Tmsf(\cluster^0_a)$ should be closed under $1-x$ and contain $0$, $\frac{1}{2}$, and $1$. Thus, $\real(t_1(\cluster^0_a))=\tfrac{1}{2}$. But as we have $w_3{:}p\geqslant t_1(\cluster^0_a)$, there is a~contradiction. Similarly, for $\times_B$, $w_3{:}p\geqslant d_2$ and $w_3{:}p\geqslant1-d_2$ imply that $v(p,w_3)\geq\tfrac{1}{2}$ which again contradicts $w_3{:}p<d_1$ and $w_3{:}p\leqslant1-d_1$.

Let us build an $\eF$-model that realises $\frownie_A$. We have three states~--- $w_0$, $w_1$, and $w_2$ --- joined into the following clusters: $\cluster^0_a=\{w_0,w_2\}$, $\cluster^0_b=\{w_0,w_1\}$, $\cluster^1_a=\{w_1\}$, and $\cluster^1_b=\{w_2\}$. We also have $\Tmsf(\cluster^0_a)=\{t_0(\cluster^0_a),t^s_0(\cluster^0_a),\one\}$ and $\Tmsf(\cluster^0_b)=\{t(\cluster^0_b),t^s(\cluster^0_b)\}$ while $\Tmsf(\cluster^1_a)$ and $\Tmsf(\cluster^1_b)$ are not specified. For simplicity, we put $\real(w_i)\!=\!w_i$ with $i\!\in\!\{0,1,2\}$. It is clear that $\real(t_0(\cluster^0_a))\!=\!0$ and \mbox{$\real(t^s_0(\cluster^0_a))\!=\!\tfrac{1}{2}$}. As for the realisation of $\Tmsf(\cluster^0_b)$, we have $\real(t(\cluster^0_b))=\frac{1}{2}$ and $\real(t^s(\cluster^0_b))=1$ since $t(\cluster^0_b)\geqslant d_0$ and $w_0{:}\invol\Box_b\invol p<d_0$ occur on the branch. Finally, we have $v(p,w_2)=1$ because of $w_2{:}p\geqslant\one$; $v(p,w_0)\geq\tfrac{1}{2}$ because of $w_0{:}p\geqslant d_1$ and $w_0{:}p\geqslant1-d_1$; $v(p,w_1)\in(0,\tfrac{1}{2}]$. A~model can be seen in Fig.~\ref{fig:countermodelexample}.

A~straightforward check gives us that $v(\Box_a(\invol p\!\rightarrow\!p),w_0)\!=\!1$, \mbox{$v(\invol\Box_a\!\invol p,w)\!=\!1$} but $v(\invol\Box_b\invol p,w_0)=\frac{1}{2}$. Thus, $v(\phi_\mathsf{plaus},w_0)=\frac{1}{2}<1$.
\end{example}
\begin{figure}
\centering
\resizebox{1\linewidth}{!}{
\begin{forest}
smullyan tableaux
[w_0\!:\!(\Box_a(\invol p\!\rightarrow\!p)\!\wedge\!\invol\Box_a\invol p)\!\rightarrow\!\invol\Box_b\invol p\!<\!1\quad w_0\in\cluster^0_a\quad w_0\in\cluster^0_b[w_0\!:\!\Box_a(\invol p\!\rightarrow\!p)\!\wedge\!\invol\Box_a\invol p\!\geqslant\!d_0\quad d_0\!\leqslant\!1[w_0\!:\!\invol\Box_b\invol p\!<\!d_0[{w_0\!:\!\Box_b\invol p\!>\!1\!-\!d_0}[w_0\!:\!\Box_a(\invol p\!\rightarrow\!p)\!\geqslant\!d_0[w_0\!:\!\invol\Box_a\invol p\!\geqslant\!d_0[w_0\!:\!\Box_a\invol p\!\leqslant\!1\!-\!d_0
[{1\!-\!d_0\!<\!1}[{w_0\!:\!\Box_b\invol p\!=\!\one}[w_0\!:\!\invol p\!\geqslant\!\one[{{w_0\!:\!p\!\leqslant\!1\!-\!\one}}
[{w_0\!:\!\Box_a(\invol p\!\rightarrow\!p)\!=\!\one}[d_0\!\leqslant\!\one[w_0\!:\!\invol p\!\rightarrow\!p\!\geqslant\!\one
[{w_0\!:\!p\!\geqslant\!\one},closed][{w_0\!:\!p\!\geqslant\!d_1}[w_0\!:\!\invol p\!\leqslant\!d_1[d_1\!\leqslant\!\one[{w_0\!:\!p\!\geqslant\!1\!-\!d_1},closed]]]]
]]][{w_0\!:\!\Box_a(\invol p\!\rightarrow\!p)\!=\!t(\cluster^0_a)}[{t(\cluster^0_a)\!\geqslant\!d_0}[w_1\!:\!\invol p\!\rightarrow\!p\!<\!t^s(\cluster^0_a)[w_1\!\in\!\cluster^0_a\quad w_1\!\in\!\cluster^1_b[w_0\!:\!\invol p\!\rightarrow\!p\!\geqslant\!t(\cluster^0_a)[w_1\!:\!\invol p\!\rightarrow\!p\!\geqslant\!t(\cluster^0_a)
[{w_0\!:\!p\!\geqslant\!t(\cluster^0_a)},closed][t(\cluster^0_a)\!\leqslant\!1[{w_0\!:\!p\!\geqslant\!d_1}[w_0\!:\!\invol p\!\leqslant\!d_1[{w_0\!:\!p\!\geqslant\!1\!-\!d_1},closed]]]]
]]]]]]]]]]
[{w_0\!:\!\Box_b\invol p\!=\!t(\cluster^0_b)}[d_0\!\leqslant\!t(\cluster^0_b)[w_1\!:\!\invol p\!<\!t^s(\cluster^0_b)[w_1\!\in\!\cluster^0_b\quad w_1\!\in\!\cluster^1_a[w_0\!:\!\invol p\!\geqslant\!t(\cluster^0_b)[w_1\!:\!\invol p\!\geqslant\!t(\cluster^0_b)[w_0\!:\!p\!\leqslant\!1\!-\!t(\cluster^0_b)[w_1\!:\!p\!\leqslant\!1\!-\!t(\cluster^0_b)[w_1\!:\!p\!>\!1\!-\!t^s(\cluster^0_b)[{1\!-\!d_0\!\geqslant\!1},closed][t_0(\cluster^0_a)\!\leqslant\!1\!-\!d_0[w_2\!:\!\invol p\!<\!t^s_0(\cluster^0_a)[w_2\!\in\!\cluster^0_a~w_2\!\in\!\cluster^1_b[w_2\!:\!p\!>\!1\!-\!t^s_0(\cluster^0_a)[w_0\!:\!\invol p\!\geqslant\!t_0(\cluster^0_a)[w_2\!:\!\invol p\!\geqslant\!t_0(\cluster^0_a)[w_0\!:\!p\!\leqslant\!1\!-\!t_0(\cluster^0_a)[w_2\!:\!p\!\leqslant\!1\!-\!t_0(\cluster^0_a)
)[{w_0\!:\!\Box_a(\invol p\!\rightarrow\!p)\!=\!\one}[d_0\!\leqslant\!1[w_0\!:\!\invol p\!\rightarrow\!p\!\geqslant\!\one[w_2\!:\!\invol p\!\rightarrow\!p\!\geqslant\!\one
[w_0\!:\!p\!\geqslant\!\one,closed][d_1\!\leqslant\!1[w_0\!:\!\invol p\!\leqslant\!d_1[w_0\!:\!p\!\geqslant\!d_1[w_0\!:\!p\!\leqslant\!1\!-\!d_1
[w_2\!:\!p\!\geqslant\!\one[\frownie_A]][d_2\!\leqslant\!1[w_2\!:\!\invol p\!\leqslant\!d_2[w_2\!:\!p\!\geqslant\!d_2[w_2\!:\!p\!\leqslant\!1\!-\!d_2[\frownie_B]]]]]
]]]]]]]][{w_0\!:\!\Box_a(\invol p\!\rightarrow\!p)\!=\!t_1(\cluster^0_a)}[d_0\!\leqslant\!t_1(\cluster^0_a)[w_3\!:\!\invol p\!\rightarrow\!p\!<\!t^s_1(\cluster^0_a)[w_3\!\in\!\cluster^0_a~w_3\!\in\!\cluster^2_b[w_0\!:\!\invol p\!\rightarrow\!p\!\geqslant\!t_1(\cluster^0_a)[w_2\!:\!\invol p\!\rightarrow\!p\!\geqslant\!t_1(\cluster^0_a)[w_3\!:\!\invol p\!\rightarrow\!p\!\geqslant\!t_1(\cluster^0_a)
[w_3\!:\!\invol p\!\geqslant\!d_1~d_1\!\leqslant\!t^s_1(\cluster^0_a)[w_3\!:\!p\!<\!d_1[w_3\!:\!p\!\leqslant\!1\!-\!d_1[w_3\!:\!p\!\geqslant\!t_1(\cluster^0_a)[\times_A]][d_2\!\leqslant\!1[w_3\!:\!\invol p\!\leqslant\!d_2[w_3\!:\!p\!\geqslant\!d_2[w_3{:}p\!\geqslant\!1\!-\!d_2[\times_B]]]]]]]]]]]]]]]]
]]]]]]]]]]]]]]]]
]]]]]]]
\end{forest}}\\
\caption{A~failed proof of $(\Box_a(\invol p\!\rightarrow\!p)\!\wedge\!\invol\Box_a\invol p)\!\rightarrow\!\invol\Box_b\invol p$. $\frownie$'s denote complete open branches.}
\label{fig:tableauxexamplepart1}
\end{figure}
\begin{figure}
\centering
\begin{align*}
\xymatrix{w_1:p=\tfrac{1}{2}~\ar@{-}|{R_b}[rr]&&~w_0:p=\tfrac{1}{2}~\ar@{-}|{R_a}[rr]&&~w_2:p=1}\\[.3em]
\forall w:T_a(w)=T_b(w)=\{0,\tfrac{1}{2},1\}
\end{align*}
\caption{A~countermodel for $\frownie_A$ (cf.~Fig.~\ref{fig:tableauxexamplepart1}). Reflexive arrows are not shown.}
\label{fig:countermodelexample}
\end{figure}

We are now ready to obtain the completeness of tableaux w.r.t.\ $\eF$-models. The proof is standard and follows~\cite[Theorems~6.9 and~6.11]{Rogger2016phd} and~\cite[Theorem~2]{BilkovaFergusonKozhemiachenko2025TARK}.
\begin{restatable}{theorem}{TsfiveinvGcompleteness}\label{theorem:TS5invGcompleteness}
$\phi\in\modalLinv$ is $\eF$-valid iff it has a~$\TcrispsfiveinvG$ proof.
\end{restatable}
\section{Complexity\label{sec:complexity}}
In the previous section, we constructed a~tableaux calculus that allows us to check whether a~given formula is valid on $\eF$-models. Let us now produce a~polynomial space decision procedure on its basis. We adapt the proof of the $\pspace$-completeness of $\KinvG$ from~\cite{BilkovaFergusonKozhemiachenko2025TARK}.%
\begin{restatable}{theorem}{complexity}\label{theorem:complexity}~
\begin{enumerate}[noitemsep,topsep=1pt]
\item Let $|\Amsf|\geq2$. Then it is $\pspace$-complete to decide whether $\phi\in\modalLinv(\Amsf)$ is valid on $\eF$-models for~$\Amsf$.
\item Let $|\Amsf|=1$. Then it is $\conp$-complete to decide whether $\phi\in\modalLinv(\Amsf)$ is valid on $\eF$-models for~$\Amsf$.
\end{enumerate}
\end{restatable}
\begin{proof}

We show Item~1 only (cf.\ the appendix for Item~2). Consider first $\pspace$-hardness. We provide a~polynomial-time reduction from the validity in classical $\sfive$ which is known to be $\conp$-complete if $|\Amsf|=1$ and $\pspace$-complete if $|\Amsf|\geq2$~\cite{HalpernMoses1992}. Let $\phi\in\modalLinv$, and denote $\phi^\triangle$ the result of replacing every occurrence $p$ of every propositional variable $p$ with $\triangle p$ (recall Convention~\ref{conv:connectives} for the definition and semantics of~$\triangle$). It is clear from Convention~\ref{conv:connectives} that $\lmc(\phi^\triangle)=\Omc(\lmc(\phi))$.

Now let $\Ffrak=\langle W,\langle R_a\rangle_{a\in\Amsf}\rangle$ be an epistemic frame. Let further, $v$~be a~\emph{classical} valuation on~$\Ffrak$ and denote $\Mfrak=\langle W,\langle R_a\rangle_{a\in\Amsf},v\rangle$. We define an $\eF$-model $\Mfrak^\triangle=\langle W,\langle R_a\rangle_{a\in\Amsf},\langle T^\triangle_a\rangle_{a\in\Amsf},v^\triangle\rangle$ as follows: $T^\triangle_a(w)=\{0,\tfrac{1}{2},1\}$ for every $w\in W$ and $a\in\Amsf$; $v^\triangle(p,w)=v(p,w)$. We can now check by induction on $\phi$ that $v^\triangle(\phi^\triangle,w)=v(\phi,w)$ for every $\phi\in\modalLinv$ and $w\in W$.

The basis case of $\phi=p$ and $\phi^\triangle=\triangle p$ is evident from~\eqref{equ:connectives}. The cases of propositional connectives $\invol$, $\wedge$, and $\rightarrow$ follow by a~straightforward application of the induction hypothesis since propositional connectives behave classically on $\{0,1\}$, $\triangle$ behaves trivially on $\{0,1\}$, and $v^\triangle$ assigns only~$0$ and~$1$. Thus, it suffices to check that $v^\triangle(\phi^\triangle,w)=1$ iff $v(\phi,w)=1$. Consider now $\phi=\Box_a\chi$ and $\phi^\triangle=\Box_a\chi^\triangle$. We have
\begin{align*}
v(\Box_a\chi,w)=1&\text{ iff }\forall w':wR_aw'\Rightarrow v(\chi,w')=1\\
&\text{ iff }\forall w':wR_aw'\Rightarrow v^\triangle(\chi^\triangle,w')=1\tag{by IH}\\
&\text{ iff }v^\triangle(\Box_a\chi^\triangle)=1
\end{align*}

For the converse direction, let $\Mfrak=\langle W,\langle R_a\rangle_{a\in\Amsf},\langle T_a\rangle_{a\in\Amsf},v\rangle$ be an $\eF$-model. We define a~classical $\sfive$-model $\Mfrak=\langle W,\langle R_a\rangle_{a\in\Amsf},v^\cl\rangle$ as follows: $v^\cl(p,w)=1$ if $v(p,w)=1$ and $v^\cl(p,w)=0$, otherwise. Again, by induction on $\phi$, we show that $v(\phi^\triangle,w)=v^\cl(\phi,w)$. First, observe that since every variable of $\phi^\triangle$ is in the scope of a~$\triangle$, $v(\phi^\triangle,w)\in\{0,1\}$ in every $w\in W$. Now, let $\phi=p$ and $\phi^\triangle=\triangle p$. If $v(\triangle p,w)=1$, then $v(p,w)=1$ whence, $v^\cl(p,w)=1$. If $v(\triangle p,w)=0$, then $v(p,w)<1$, whence, $v^\cl(p,w)=0$. The cases of propositional connectives can be dealt with by direct applications of the induction hypothesis. Finally, let $\phi=\Box_a\chi$ and $\phi^\triangle=\Box_a\chi^\triangle$. If $v(\Box_a\chi^\triangle,w)=1$, then $v(\chi^\triangle,w')=1$ in every $w'$ s.t.\ $wR_aw'$. Thus, $v^\cl(\chi,w')=1$ in all such $w'$'s by the induction hypothesis, whence $v^\cl(\Box_a\chi,w)=1$. If $v(\Box_a\chi^\triangle,w)=0$, then recall that $\chi^\triangle$ can have only values in $\{0,1\}$. Thus, there is some $w'$ s.t.\ $wR_aw'$ and $v(\chi^\triangle,w')=0$. Again, by the induction hypothesis, we obtain $v^\cl(\chi,w')=0$, whence, $v^\cl(\Box_a\chi,w)=0$, as required.

For the $\pspace$-membership, we begin by observing that \emph{every} tableau terminates. Indeed, all rules except for $\Box_\triangleright$ are finitely branching and decompose the formulas in the premise. If $\Box_\triangleright$ is applied, then a~constraint of the form $w{:}\Box_a\psi=\Tfrak$ is introduced to which we can apply $\Box_=$ which does decompose the formula. Moreover, as we have noticed in Section~\ref{sec:tableaux}, if a~branch contains several instances of a~modal constraint in different states of the same cluster, it suffices to apply a~modal rule to only one such instance.

For simplicity, we assume that $\Amsf=\{a,b\}$. Decision procedures for larger $\Amsf$'s can be obtained in a~similar manner. The algorithm runs as follows: given $\phi$, we start building a~tableau for $\Tmc(\phi)=\{w{:}\phi<1,w\in\cluster^0_a,w\in\cluster^0_b\}$. If a~rule introduces branching, we pick one branch of the tableau and work with it depth-first. If the branch we are working with is closed, we delete it and choose the next one. Our goal is to build a~model realising $\Tmc(\phi)$ ‘on the fly’ (cf.~Fig.~\ref{fig:algorithm} for an illustration). As we proceed depth-first on the tableau, we denote the branch we are working on with~$\Bmc$.

We begin with applying propositional rules until all labelled formulas in~$\Bmc$ have $\Box$'s as principal connectives. Once all propositional rules are applied, we pick $a\in\Amsf$ and apply modal rules for $\Box_a$-formulas. We begin with $\Box_\triangleright$, $\Box_\leqslant$, and $\Box_<$ rules. After all such rules are applied, we apply $\equiv_\cluster$ rules and propositional rules if needed. This adds $\Omc(\lmc(\phi))$ states to the $R_a$-cluster $\cluster^0_a$. Then we apply $\Box_=$ rules using the generated states. Note that $\Box_=$ uses the states generated by the applications of other modal rules. Moreover, as we apply all instances of $\Box_=$ in~$\cluster^0_a$ \emph{at once}, we have to first apply all rules $\Box_\triangleright$, $\Box_\leqslant$, $\Box_<$, and $\equiv_\cluster$ in~$\cluster^0_a$ (otherwise, some instances of $\Box_=$ might not be applicable). Note that $\Box_=$ is the only modal rule that can become applicable after the application of $\Box_\triangleright$. Observe also that we avoid generating new (proper) clusters until all modal rules are applied in $\cluster^0_a$.

Once done, we pick one state (say, $u$) in $\cluster^0_a$ s.t.\ $u:\Box_b\psi\triangledown\Tfrak$ is present on~$\Bmc$ and mark $u$ as ‘active’. We then apply modal rules for $\Box_b$ in the described manner but to $u$ only. This generates an $R_b$-cluster (say, $\cluster^1_b$) that also contains at most $\lmc(\phi)$ states. In this cluster, we pick a~state $u'$ s.t.\ $u':\Box_a\psi'\triangledown\Tfrak'\in\Bmc$, mark $u'$ ‘active’, and repeat the process until we obtain a~cluster that \emph{does not} contain labelled formulas of the form $\Box_a\chi\triangledown\Zmc$ (say, $\cluster'_b$ that was generated by a~state $s$ in an $R_a$-cluster $\cluster'_a$). We will generate at most $\lmc(\phi)$ clusters: there will be as many clusters as alternations of agent labels of nested modalities\footnote{E.g., it suffices to use one cluster for $\Box_a\Box_ap$, but we need \emph{two} clusters for $\Box_a\Box_bq$.} and each of these will have at most $\lmc(\phi)$ states. We check whether $\Bmc$ is closed. Note that a~model satisfying~$\Bmc$ can be guessed in non-deterministic polynomial time.

If $\Bmc$ is open, we mark every state $u''$ s.t.\ $u''\in\cluster'_b$ occurs on~$\Bmc$ as ‘safe’. Then we delete all labelled formulas of the form $u':\chi\triangledown\Zmc$ s.t.\ $u'\in\cluster'_b$ occurs on~$\Bmc$ from~$\Bmc$. Furthermore, we delete all non-active states in $\cluster'_b$ and remove the ‘active’ mark from~$s$. Then we pick another state in $\cluster'_a$ and repeat the process. The algorithm runs until either $w$ is marked as ‘safe’ (whence, $\phi$ is not valid) or all branches of the tableau are closed (in which case, $\phi$ is valid). Since, as we have remarked earlier, all tableaux terminate, we will stop at some point. Note furthermore, that at any given time the branch of the tableau contains at most $\lmc(\phi)$ proper clusters each with at most $\lmc(\phi)$ states (i.e., $\Omc(\lmc(\phi)^2)$ states in total). Each state contains subformulas of $\phi$ occurring in constraints, constraints consisting of two value terms, and relational terms indicating to which clusters it belongs. Thus, we need $\Omc(\lmc(\phi)^3)$ space to execute the algorithm.
\end{proof}
\begin{figure}[!ht]
\centering
\resizebox{.8\textwidth}{!}{%
\begin{circuitikz}
\tikzstyle{every node}=[font=\LARGE]
\draw [ color={rgb,255:red,0; green,255; blue,0} , fill={rgb,255:red,0; green,255; blue,0}] (18,-0.25) circle (0.25cm);
\draw [ color={rgb,255:red,255; green,0; blue,0} , line width=2pt , rotate around={122:(6.5,5.5)}] (6.5,5.5) ellipse (1.75cm and 0.75cm);
\draw [ color={rgb,255:red,0; green,0; blue,255} , line width=2pt , rotate around={128:(5,7.5)}] (5,7.5) ellipse (1.75cm and 0.75cm);
\draw [ color={rgb,255:red,0; green,0; blue,255} , line width=2pt , rotate around={128:(21.25,14)}] (21.25,14) ellipse (1.75cm and 0.75cm);
\draw [ color={rgb,255:red,255; green,0; blue,0} , line width=2pt ] (12.75,15.25) ellipse (2.25cm and 1cm);
\draw [ color={rgb,255:red,255; green,0; blue,255}, line width=2pt, ->, >=Stealth] (1.5,15.25) -- (3.25,15.25);
\node [font=\LARGE] at (2,15) {};
\node [font=\LARGE] at (2.25,15) {};
\draw [ fill={rgb,255:red,0; green,0; blue,0} ] (5,15.5) circle (0.25cm);
\draw [ fill={rgb,255:red,0; green,0; blue,0} ] (5,14.75) circle (0.25cm);
\draw [ fill={rgb,255:red,0; green,0; blue,0} ] (6.5,15.5) circle (0.25cm);
\draw [ fill={rgb,255:red,0; green,0; blue,0} ] (6.5,14.75) circle (0.25cm);
\draw [ color={rgb,255:red,255; green,0; blue,0} , line width=2pt ] (5.75,15.25) ellipse (2.25cm and 1cm);
\draw [ color={rgb,255:red,255; green,0; blue,255}, line width=2pt, ->, >=Stealth] (8.5,15.25) -- (10.25,15.25);
\draw [ fill={rgb,255:red,0; green,0; blue,0} ] (12,15.5) circle (0.25cm);
\draw [ fill={rgb,255:red,0; green,0; blue,0} ] (12,14.75) circle (0.25cm);
\draw [ fill={rgb,255:red,0; green,0; blue,0} ] (13.5,15.5) circle (0.25cm);
\draw [ line width=2pt ] (13.5,14.75) circle (0.25cm);
\draw [ color={rgb,255:red,0; green,0; blue,255} , line width=2pt , rotate around={128:(14.25,13.75)}] (14.25,13.75) ellipse (1.75cm and 0.75cm);
\draw [ fill={rgb,255:red,0; green,0; blue,0} ] (14.75,13) circle (0.25cm);
\draw [ fill={rgb,255:red,0; green,0; blue,0} ] (14.25,13.75) circle (0.25cm);
\node [font=\LARGE] at (2.5,14.5) {};
\draw [ color={rgb,255:red,255; green,0; blue,255}, line width=2pt, ->, >=Stealth] (15.75,15.5) -- (17.5,15.5);
\draw [ color={rgb,255:red,255; green,0; blue,0} , line width=2pt ] (20,15.5) ellipse (2.25cm and 1cm);
\draw [ fill={rgb,255:red,0; green,0; blue,0} ] (19,15.75) circle (0.25cm);
\draw [ fill={rgb,255:red,0; green,0; blue,0} ] (19,15) circle (0.25cm);
\draw [ fill={rgb,255:red,0; green,0; blue,0} , line width=0.2pt ] (20.5,15.75) circle (0.25cm);
\draw [ line width=2pt ] (20.5,15) circle (0.25cm);
\draw [ line width=2pt ] (21.75,13.25) circle (0.25cm);
\draw [ fill={rgb,255:red,0; green,0; blue,0} ] (21.25,14) circle (0.25cm);
\node [font=\LARGE] at (1.5,14.75) {};
\draw [ fill={rgb,255:red,0; green,0; blue,0} ] (6.5,5.5) circle (0.25cm);
\node [font=\LARGE] at (1.5,14.75) {};
\draw [ fill={rgb,255:red,0; green,0; blue,0} , line width=0.2pt ] (7,4.75) circle (0.25cm);
\draw [ color={rgb,255:red,255; green,0; blue,255}, line width=2pt, ->, >=Stealth, dashed] (19,11.75) -- (6,10);
\draw [ color={rgb,255:red,255; green,0; blue,0} , line width=2pt ] (3.5,8.75) ellipse (2.25cm and 1cm);
\draw [ fill={rgb,255:red,0; green,0; blue,0} ] (2.75,9.25) circle (0.25cm);
\draw [ fill={rgb,255:red,0; green,0; blue,0} ] (2.75,8.5) circle (0.25cm);
\draw [ fill={rgb,255:red,0; green,0; blue,0} ] (4.25,9.25) circle (0.25cm);
\draw [ line width=2pt ] (4.25,8.5) circle (0.25cm);
\draw [ line width=2pt ] (5.75,6.5) circle (0.25cm);
\draw [ fill={rgb,255:red,0; green,0; blue,0} ] (5.25,7.25) circle (0.25cm);
\draw [ color={rgb,255:red,255; green,0; blue,0} , line width=2pt , rotate around={122:(22.25,12.25)}] (22.25,12.25) ellipse (1.75cm and 0.75cm);
\node [font=\LARGE] at (1.5,14.75) {};
\draw [ color={rgb,255:red,0; green,255; blue,0} , fill={rgb,255:red,0; green,255; blue,0}] (8.75,-0.75) circle (0.25cm);
\draw [ color={rgb,255:red,0; green,255; blue,0} , line width=2pt ] (7.25,0) circle (0.25cm);
\draw [ color={rgb,255:red,0; green,255; blue,0} , fill={rgb,255:red,0; green,255; blue,0}] (8,-1.5) circle (0.25cm);
\draw [ color={rgb,255:red,0; green,0; blue,255} , line width=2pt , rotate around={34:(8,-1)}] (8,-1) ellipse (1.25cm and 2cm);
\draw [ color={rgb,255:red,255; green,0; blue,0} , line width=2pt , rotate around={-141:(6.5,0.75)}] (6.5,0.75) ellipse (0.75cm and 2cm);
\draw [ fill={rgb,255:red,0; green,0; blue,0} ] (6.25,1) circle (0.25cm);
\draw [ line width=2pt ] (5.75,1.75) circle (0.25cm);
\draw [ color={rgb,255:red,255; green,0; blue,255}, line width=2pt, ->, >=Stealth] (10.5,4.25) -- (12.25,4.25);
\draw [line width=2pt, dashed] (5.25,4.25) -- (5.25,2.75);
\draw [ color={rgb,255:red,255; green,0; blue,0} , line width=2pt ] (13.75,8.5) ellipse (2.25cm and 1cm);
\draw [ color={rgb,255:red,0; green,0; blue,255} , line width=2pt , rotate around={128:(15.25,7)}] (15.25,7) ellipse (1.75cm and 0.75cm);
\draw [ color={rgb,255:red,255; green,0; blue,0} , line width=2pt , rotate around={122:(16.5,5.25)}] (16.5,5.25) ellipse (1.75cm and 0.75cm);
\draw [line width=2pt, dashed] (16,4) -- (16,2.5);
\draw [ color={rgb,255:red,255; green,0; blue,0} , line width=2pt , rotate around={-141:(17.5,0.5)}] (17.5,0.5) ellipse (0.75cm and 2cm);
\node [font=\LARGE] at (1.5,14.75) {};
\draw [ fill={rgb,255:red,0; green,0; blue,0} ] (13,8.75) circle (0.25cm);
\draw [ fill={rgb,255:red,0; green,0; blue,0} ] (13,8) circle (0.25cm);
\draw [ fill={rgb,255:red,0; green,0; blue,0} ] (14.25,8.75) circle (0.25cm);
\draw [ line width=2pt ] (14.5,8) circle (0.25cm);
\draw [ fill={rgb,255:red,0; green,0; blue,0} ] (15.25,7) circle (0.25cm);
\draw [ line width=2pt ] (16,6.25) circle (0.25cm);
\draw [ fill={rgb,255:red,0; green,0; blue,0} , line width=0.2pt ] (16.75,5.25) circle (0.25cm);
\draw [ fill={rgb,255:red,0; green,0; blue,0} , line width=2pt ] (1,15.25) circle (0.25cm);
\draw [ fill={rgb,255:red,0; green,0; blue,0} ] (17,4.5) circle (0.25cm);
\draw [ fill={rgb,255:red,0; green,0; blue,0} ] (22.5,12) circle (0.25cm);
\draw [ fill={rgb,255:red,0; green,0; blue,0} ] (22.75,11.25) circle (0.25cm);
\draw [ line width=2pt ] (16.75,1.25) circle (0.25cm);
\draw [ line width=2pt ] (17.5,0.5) circle (0.25cm);
\draw [ color={rgb,255:red,0; green,0; blue,255} , line width=2pt , rotate around={-141:(16.5,-0.25)}] (16.5,-0.25) ellipse (1.75cm and 0.75cm);
\draw [ fill={rgb,255:red,0; green,0; blue,0} , line width=2pt ] (16.25,-0.25) circle (0.25cm);
\draw [ fill={rgb,255:red,0; green,0; blue,0} , line width=2pt ] (15.75,-0.75) circle (0.25cm);
\node [font=\LARGE] at (5.75,16.75) {$\cluster^0_a$};
\node [font=\LARGE] at (12.5,16.75) {$\cluster^0_a$};
\node [font=\LARGE] at (20,17) {$\cluster^0_a$};
\node [font=\LARGE] at (14.5,12) {$\cluster^1_b$};
\node [font=\LARGE] at (19.75,13.5) {$\cluster^1_b$};
\node [font=\LARGE] at (22,10.5) {$\cluster^2_a$};
\node [font=\LARGE] at (3.5,10.25) {$\cluster^0_a$};
\node [font=\LARGE] at (13.75,10) {$\cluster^0_a$};
\node [font=\LARGE] at (3.75,7) {$\cluster^1_b$};
\node [font=\LARGE] at (5.25,5) {$\cluster^2_a$};
\node [font=\LARGE] at (6.75,2.25) {$\cluster'_a$};
\node [font=\LARGE] at (8.5,-3.25) {$\cluster'_b$};
\node [font=\LARGE] at (13.75,6.75) {$\cluster^1_b$};
\node [font=\LARGE] at (15.25,5) {$\cluster^2_a$};
\node [font=\LARGE] at (17.75,2) {$\cluster'_a$};
\node [font=\LARGE] at (15.75,-2) {$\cluster''_b$};
\end{circuitikz}
}%
\caption{Building an $\eF$-model from a~tableau proof on the fly: ellipses denote proper clusters; active states are marked with $\bigcirc$'s; safe states are \textcolor{green}{green}.}
\label{fig:algorithm}
\end{figure}
\section{Equivalence with the standard semantics\label{sec:equivalence}}
In Sections~\ref{sec:tableaux} and~\ref{sec:complexity}, we provided a~sound and complete tableaux calculus for $\crispsfiveinvG$ and obtained the complexity of validity. We can also obtain the finite model property for the $\eF$-model based semantics.
\begin{corollary}\label{cor:FMP}
A~formula $\phi\in\modalLinv$ is valid on all $\eF$-models iff it is valid on all \emph{finite} $\eF$-models.
\end{corollary}
\begin{proof}
As $\TcrispsfiveinvG$ is sound and complete and each complete open branch in the failed proof of~$\phi$ contains a~finite $\eF$-countermodel of~$\phi$, the statement follows.
\end{proof}
It remains to show that the sets of formulas valid on $\crispsfiveinvG$-models and $\eF$-models coincide. To do this, we adapt the approaches of~\cite{CaicedoMetcalfeRodriguezRogger2013} and~\cite{BilkovaFergusonKozhemiachenko2025TARK}.

We begin with a~special version of the tree model property. Namely, we show that it suffices to check whether a~formula is satisfiable on $\eF$-models whose frame is a~tree of proper clusters (cf.~Fig.~\ref{fig:clustertree} for an illustration).
\begin{definition}\label{def:clustertreeframe}
Let $\Ffrak=\langle W,\langle R_a\rangle_{a\in\Amsf}\rangle$ be an epistemic frame for~$\Amsf$. Given two proper clusters $\cluster$ and $\cluster'$ in~$\Ffrak$, we write $\cluster\between\cluster'$ if $\cluster\neq\cluster'$ and $|\cluster\cap\cluster'|=1$.

We say that $\Ffrak$ is a~\emph{cluster tree} if it holds that:
\begin{itemize}[noitemsep,topsep=1pt]
\item any two distinct proper clusters have at most one state in common;
\item $\langle\ClusterSet(\Ffrak),\between\rangle$ is a~tree.
\end{itemize}

If we designate $\cluster$ as the \emph{root} of~$\Ffrak$, we say that the \emph{height of~$\Ffrak$} ($\Hmc(\Ffrak)$) is the number of proper clusters in the longest branch originating from $\cluster$.

An epistemic ($\Fmsf$-)model $\Mfrak$ whose underlying frame is a~cluster tree is called an \emph{epistemic cluster tree ($\Fmsf$-)model}. The height of $\Mfrak$ ($\Hmc(\Mfrak)$) is the height of its underlying frame.
\end{definition}
\begin{figure}
\centering
$\Ffrak$
\resizebox{.375\textwidth}{!}{%
\begin{circuitikz}
\tikzstyle{every node}=[font=\LARGE]
\draw [ color={rgb,255:red,255; green,0; blue,0} , line width=2pt ] (11.75,15) ellipse (1.75cm and 1cm);
\draw [ fill={rgb,255:red,0; green,0; blue,0} , line width=2pt ] (18.25,8.5) circle (0cm);
\draw [ fill={rgb,255:red,0; green,0; blue,0} , line width=2pt ] (11.75,15.25) circle (0.25cm);
\draw [ fill={rgb,255:red,0; green,0; blue,0} , line width=2pt ] (10.75,14.75) circle (0.25cm);
\draw [ fill={rgb,255:red,0; green,0; blue,0} , line width=2pt ] (9.5,14.5) circle (0.25cm);
\draw [ fill={rgb,255:red,0; green,0; blue,0} , line width=2pt ] (9,14) circle (0cm);
\draw [ color={rgb,255:red,255; green,0; blue,0} , line width=2pt , rotate around={-128:(6.75,12)}] (6.75,12) ellipse (2.25cm and 0.75cm);
\draw [ color={rgb,255:red,0; green,0; blue,255} , line width=2pt , rotate around={-153:(9,13.75)}] (9,13.75) ellipse (2.5cm and 1cm);
\draw [ fill={rgb,255:red,0; green,0; blue,0} , line width=2pt ] (7.75,13.25) circle (0.25cm);
\draw [ fill={rgb,255:red,0; green,0; blue,0} , line width=2pt ] (9.25,13.5) circle (0.25cm);
\draw [ color={rgb,255:red,255; green,0; blue,0} , line width=2pt , rotate around={-47:(10,12.75)}] (10,12.75) ellipse (1.5cm and 0.75cm);
\draw [ fill={rgb,255:red,0; green,0; blue,0} , line width=2pt ] (12.75,14.75) circle (0.25cm);
\draw [ fill={rgb,255:red,0; green,0; blue,0} , line width=2pt ] (10.5,12.25) circle (0.25cm);
\draw [ fill={rgb,255:red,0; green,0; blue,0} , line width=2pt ] (5.75,10.75) circle (0.25cm);
\draw [ fill={rgb,255:red,0; green,0; blue,0} , line width=2pt ] (6.5,11.75) circle (0.25cm);
\draw [ fill={rgb,255:red,0; green,0; blue,0} , line width=2pt ] (13.75,14) circle (0.25cm);
\draw [ fill={rgb,255:red,0; green,0; blue,0} , line width=2pt ] (14.5,13.25) circle (0.25cm);
\draw [ fill={rgb,255:red,0; green,0; blue,0} , line width=2pt ] (15.5,12.25) circle (0.25cm);
\draw [ fill={rgb,255:red,0; green,0; blue,0} , line width=2pt ] (16.5,11.5) circle (0.25cm);
\draw [ fill={rgb,255:red,0; green,0; blue,0} , line width=2pt ] (11.5,10.75) circle (0.25cm);
\draw [ color={rgb,255:red,0; green,0; blue,255} , line width=2pt , rotate around={-216:(13.75,14)}] (13.75,14) ellipse (2cm and 0.75cm);
\draw [ color={rgb,255:red,255; green,0; blue,0} , line width=2pt , rotate around={-41:(15.5,12.25)}] (15.5,12.25) ellipse (2cm and 0.75cm);
\draw [ color={rgb,255:red,0; green,0; blue,255} , line width=2pt , rotate around={-53:(11,11.5)}] (11,11.5) ellipse (1.5cm and 0.75cm);
\end{circuitikz}}
$\Gfrak$
\resizebox{.375\linewidth}{!}{
\begin{circuitikz}
\tikzstyle{every node}=[font=\LARGE]
\draw [ color={rgb,255:red,255; green,0; blue,0} , line width=2pt ] (11.75,15) ellipse (1.75cm and 1cm);
\draw [ fill={rgb,255:red,0; green,0; blue,0} , line width=2pt ] (18.25,8.5) circle (0cm);
\draw [ fill={rgb,255:red,0; green,0; blue,0} , line width=2pt ] (11.75,15.25) circle (0.25cm);
\draw [ fill={rgb,255:red,0; green,0; blue,0} , line width=2pt ] (10.75,14.75) circle (0.25cm);
\draw [ fill={rgb,255:red,0; green,0; blue,0} , line width=2pt ] (9.5,14.5) circle (0.25cm);
\draw [ fill={rgb,255:red,0; green,0; blue,0} , line width=2pt ] (9,14) circle (0cm);
\draw [ color={rgb,255:red,255; green,0; blue,0} , line width=2pt , rotate around={-128:(6.75,12)}] (6.75,12) ellipse (2.25cm and 0.75cm);
\draw [ color={rgb,255:red,0; green,0; blue,255} , line width=2pt , rotate around={-153:(9,13.75)}] (9,13.75) ellipse (2.5cm and 1cm);
\draw [ fill={rgb,255:red,0; green,0; blue,0} , line width=2pt ] (7.75,13.25) circle (0.25cm);
\draw [ fill={rgb,255:red,0; green,0; blue,0} , line width=2pt ] (9.25,13.5) circle (0.25cm);
\draw [ color={rgb,255:red,255; green,0; blue,0} , line width=2pt , rotate around={-47:(10,12.75)}] (10,12.75) ellipse (1.5cm and 0.75cm);
\draw [ fill={rgb,255:red,0; green,0; blue,0} , line width=2pt ] (12.75,14.75) circle (0.25cm);
\draw [ fill={rgb,255:red,0; green,0; blue,0} , line width=2pt ] (10.5,12.25) circle (0.25cm);
\draw [ fill={rgb,255:red,0; green,0; blue,0} , line width=2pt ] (5.75,10.75) circle (0.25cm);
\draw [ fill={rgb,255:red,0; green,0; blue,0} , line width=2pt ] (6.5,11.75) circle (0.25cm);
\draw [ fill={rgb,255:red,0; green,0; blue,0} , line width=2pt ] (13.75,14) circle (0.25cm);
\draw [ fill={rgb,255:red,0; green,0; blue,0} , line width=2pt ] (14.5,13.25) circle (0.25cm);
\draw [ fill={rgb,255:red,0; green,0; blue,0} , line width=2pt ] (15.5,12.25) circle (0.25cm);
\draw [ fill={rgb,255:red,0; green,0; blue,0} , line width=2pt ] (16.5,11.5) circle (0.25cm);
\draw [ fill={rgb,255:red,0; green,0; blue,0} , line width=2pt ] (11.5,10.75) circle (0.25cm);
\draw [ color={rgb,255:red,0; green,0; blue,255} , line width=2pt , rotate around={-216:(13.75,14)}] (13.75,14) ellipse (2cm and 0.75cm);
\draw [ color={rgb,255:red,255; green,0; blue,0} , line width=2pt , rotate around={-41:(15.5,12.25)}] (15.5,12.25) ellipse (2cm and 0.75cm);
\draw [ color={rgb,255:red,0; green,0; blue,255} , line width=2pt , rotate around={-53:(11,11.5)}] (11,11.5) ellipse (1.5cm and 0.75cm);
\draw [ color={rgb,255:red,255; green,0; blue,0} , line width=2pt , rotate around={-122:(12.75,12.5)}] (12.75,12.5) ellipse (3.5cm and 0.5cm);
\end{circuitikz}
}
\caption{Two epistemic frames for $\Amsf=\{\text{red},\text{blue}\}$. $\Ffrak$ a~cluster tree; $\Gfrak$ is \textbf{not}.}
\label{fig:clustertree}
\end{figure}

The following statement now follows from Theorem~\ref{theorem:TS5invGcompleteness} since the models produced by tab\-leaux are cluster trees.
\begin{restatable}{corollary}{clustertree}\label{cor:clustertree}
For any $\phi\in\modalLinv$, it holds that $v(\phi,w)=x$ in some $\eF$-model~$\Mfrak$ and $w\!\in\!\Mfrak$ iff there are a cluster tree $\eF$-model $\Mfrak'$ and $w'\!\in\!\Mfrak'$ s.t.\ \mbox{$v(\phi,w')\!=\!x$}, $\Hmc(\Mfrak)\!=\!\Omc(\lmc(\phi))$, and the root of~$\Mfrak'$ is the proper cluster to which $w'$~belongs.
\end{restatable}

We can now use Corollary~\ref{cor:clustertree} to obtain the equivalence between the semantics. The next statement can be shown similarly to~\cite[Theorem~1]{BilkovaFergusonKozhemiachenko2025TARK}.
\begin{restatable}{theorem}{semanticsequivalence}\label{theorem:semanticsequivalence}
An $\modalLinv$-formula $\phi$ is $\crispsfiveinvG$-valid iff it is $\eF$-valid.
\end{restatable}
\begin{proof}
If $\phi$ is \emph{invalidated} in an $\crispsfiveinvG$-model~$\Mfrak$, we reuse it as an $\eF$-model. We define $T_a(w)$'s as the sets of values of $\Box_a$-subformulas of $\phi$ and their $\invol$-negations at~$w$, adding $\{0,\tfrac{1}{2},1\}$ if needed. This way, $T_a$'s will be finite and satisfy the conditions from Definition~\ref{def:epistemicFmodels}, and the values of modal formulas will be preserved. Hence, $\phi$~will be invalidated in an $\eF$-model. For the converse direction, we use the proofs of Lemmas~2 and~4 in~\cite{CaicedoMetcalfeRodriguezRogger2013} and Corollary~\ref{cor:clustertree}. Given $\phi$, a~cluster-tree $\eF$-model $\Mfrak$ and $w\in\Mfrak$ in its root cluster, we construct an $\crispsfiveinvG$-model $\widehat{\Mfrak}$ and $\widehat{w}\in\widehat{\Mfrak}$ s.t.\ $v(\phi,w)=\widehat{v}(\phi,\widehat{w})$. We consider countably infinitely many order embeddings $h^a_k:[0,1]\rightarrow[0,1]$ for each $a\in\Amsf$ with $h^a_k(0)=0$, $h^a_k(1-x)=1-h^a_k(x)$, and $h^a_k(1)=1$ that ‘squeeze’ the open intervals between members of $T_a(w)$ closer to either their lower or upper bounds. Then we take infinitely many copies of the clusters in the original $\Fmsf$-model without $T_a$'s obtained via these embeddings. The resulting infima and suprema will coincide with the next smaller or larger member of $T_a(w)$. Thus, the required values of the formulas at $w$ in the original $\Fmsf$-model will be preserved.
\end{proof}

\begin{corollary}\label{cor:complexity}~
\begin{enumerate}[noitemsep,topsep=1pt]
\item It is $\pspace$-complete to decide whether $\phi\in\modalLinv(\Amsf)$ is $\crispsfiveinvG$-valid if $|\Amsf|\geq2$.
\item It is $\conp$-complete to decide whether $\phi\in\modalLinv(\Amsf)$ is $\crispsfiveinvG$-valid if $|\Amsf|=1$.
\end{enumerate}
\end{corollary}
\section{Conclusion\label{sec:conclusion}}
In this paper, we presented an epistemic G\"{o}del logic with involutive negation~$\crispsfiveinvG$. We built its semantics with the finite model property and used it to show that $\crispsfiveinvG$-validity is $\pspace$-complete ($\conp$-complete in the single agent case). Our next steps are as follows. First, we are not aware of any axiomatisation of $\crispsfiveinvG$ (nor of any G\"{o}del modal logic with $\invol$). Thus, it makes sense to provide one. Second, we plan to expand the $\eF$-model-based semantics for the case of \emph{doxastic} G\"{o}del logics. Finally, it would be instructive to adapt the $\Fmsf$-model-based semantics for G\"{o}del modal logics with group knowledge operators such as distributed or common knowledge.
\bibliographystyle{splncs04}
\bibliography{reference}
\clearpage
\appendix
\section{Proofs of Section~\ref{sec:tableaux}}
\TsfiveinvGcompleteness*
\begin{proof}
We begin with soundness. It is clear from Definitions~\ref{def:TS5invG} and~\ref{def:realisingmodel} that if $\Bmc$ is a~closed branch, then there is no~$\Mfrak$ and $\real$ s.t.\ $\Mfrak$ realises~$\Bmc$ under~$\real$. Thus, it suffices to establish that if $\Mfrak$ realises the premise of a~rule, then it should realise at least one conclusion as well. We consider the case of the $\Box_\triangleleft$ rule as an example. In particular, we assume that the premise is as follows: $\{w{:}\Box_a\phi\leqslant\Tfrak,w\in\cluster_a\}$. Let $\Mfrak=\langle W,\langle R\rangle_{a\in\Amsf},\langle T_a\rangle_{a\in\Amsf},v\rangle$ realise $w{:}\Box_a\phi\leqslant\Tfrak$. That is, $v(\Box_a\phi,w)\leqslant\real(\Tfrak)$. If $\real(\Tfrak)=1$, then $1\leqslant\Tfrak$ (the left conclusion) is realised. Consider the case where $\real(\Tfrak)<1$. Then (cf.~Definition~\ref{def:epistemicFmodels}) $\max\{x\in T_a(w)\mid x\leq\inf\{v(\phi,w')\mid wR_aw'\}\}\leq\real(\Tfrak)$. This means that there is a~state $w'$ $R_a$-accessible from $w$ (i.e., in the same $R_a$-cluster as~$w$) and $x,x'\in T_a(w)$ s.t.\ $x'$ is the immediate successor of $x$, $v(\Box_a\phi,w)=x$, and $v(\phi,w')<x'$. Now setting $\real(t(\cluster_a))=x$ and $\real(t^s(\cluster_a))=x'$, we have that $t(\cluster_a)\leqslant\Tfrak$ and $w':\phi<t^s(\cluster_a)$ are realised. Hence, as $T_a(w)=T_a(w')$ when $wR_aw'$, the right conclusion is realised.

Let us now proceed to completeness. We will prove that for every complete open branch, there is a~model $\Mfrak$ and a~realisation $\real$ s.t.\ $\Mfrak$ realises $\Bmc$ under $\real$. Let $\Bmc$ be a~complete open branch and consider the system of inequalities $\Bmc^\tmc$ as shown in~\eqref{equ:closure} and its solution as specified in Definition~\ref{def:TS5invG}. Now define $\Mfrak$ as follows:
\begin{align*}
W&=\{w\mid w\text{ occurs in }\Bmc\}\\
wR_aw'&\text{ iff }\exists\cluster_a\in\ClusterSet,a\in\Amsf:w\in\cluster_a\text{ and }w'\in\cluster_a\text{ occur on }\Bmc\\
T_a(w)&=\{x_{t(\cluster_a)}\mid w\in\cluster_a\text{ and }t(\cluster_a)\text{ occur in }\Bmc\}\cup\\
&\quad~\{x_{t^s(\cluster_a)}\mid w\in\cluster_a\text{ and }t^s(\cluster_a)\text{ occur in }\Bmc\}\cup\{0,\tfrac{1}{2},1\}\\
v(p,w)&=x_{w{:}p}
\end{align*}
for every $w$, $w'$, and $p\in\Prop$ occurring in $\Bmc$. We also set $\real(w)=w$, $\real(\Tfrak)=x_\Tfrak$, and $\real(w{:}\chi)=x_{w{:}\chi}$ for every $w\in\WorldLabels$, $\Tfrak\in\valueterm$, and $w{:}\chi\in\LF$ occurring on~$\Bmc$. It remains to show that all constraints are realised by $\Mfrak$ under $\real$.

For all constraints of the form $\Tfrak\triangledown\Tfrak'$ that do not contain labelled formulas, we have $\real(\Tfrak)\triangledown\real(\Tfrak')$ by the construction of~$\Mfrak$ since $\Bmc$ is a~complete open branch. Consider now constraints of the form $u:\psi\triangledown\Tfrak$. We proceed by induction on formulas. If $p\in\Prop$, $w{:}p\triangledown\Tfrak$ is realised by the construction of~$\Mfrak$. The cases of constraints whose principal connective is $\invol$, $\wedge$, or $\rightarrow$ can be shown by simple applications of the induction hypothesis.

We consider $w{:}\Box_a\phi\!\leqslant\!\Tfrak$ as an example and show that \mbox{$v(\Box_a\phi,\!w)\!\leq\!\real(\Tfrak)$}. Since $w{:}\Box_a\phi\leqslant\Tfrak$ is in~$\Bmc$ and $\Bmc$ is complete and open, it also contains (a)~$1\leqslant\Tfrak$ or (b) $t(\cluster_a)\leqslant\Tfrak$, $w'{:}\phi<t^s(\cluster_a)$, $w'\in\cluster_a$, and $w'\in\cluster'_{a'}$ for every $a'\in\Amsf$. If (a) is the case, $\Tfrak$ is a~value term (observe from Definition~\ref{def:constraints} that constraints can contain at most one labelled formula), whence $1\leqslant\Tfrak$ is realised by the construction of $\Mfrak$. Thus, $\real(\Tfrak)=1$ and $v(\Box_a\phi,w)\leq\real(\Tfrak)$, as required. If (b) is the case, then by the induction hypothesis $t(\cluster_a)\leqslant\Tfrak$ and $w':\phi<t^s(\cluster_a)$ are realised, and~$w$ and~$w'$ belong to the same $R_a$-cluster. Thus, $x_{t(\cluster_a)}\leq\real(\Tfrak)$ and $v(\phi,w')<x_{t^s(\cluster_a)}$ with $\{x_{t(\cluster_a)},x_{t^s(\cluster_a)}\}\subseteq T(w)$ and $x_{t^s(\cluster_a)}$ being the immediate successor of $x_{t(\cluster_a)}$. But in this case, it is clear that $v(\Box_a\phi,w)=\max\{x\in T(w)\mid x\leq\inf\{v(\phi,w')\mid wR_aw'\}\}\leq x_{t(w)}$, whence, $v(\Box_a\phi,w)\leq\real(\Tfrak)$, as required.
\end{proof}



\section{Proofs of Section~\ref{sec:complexity}}
We supply the missing details of the proof of Theorem~\ref{theorem:complexity}. Namely, we show that the validity of $\modalLinv(\Amsf)$-formulas with $\Amsf=\{a\}$ is $\conp$-complete.
\complexity*
\begin{proof}
To see that Item~2 holds, observe that if $\Amsf=\{a\}$, we will generate only one cluster of size $\Omc(\lmc(\phi))$. Furthermore, every cluster will contain $\Omc(\lmc(\phi))$ $\Tmsf$-symbols (at most two for each $\Box$ occurring in~$\phi$). Thus, applying Theorem~\ref{theorem:TS5invGcompleteness}, we have that if $\phi$ is not valid on $\eF$-models, there is some $\eF$-model $\Mfrak$ of the size $\Omc(\lmc(\phi))$ s.t.\ $v(\phi,u)\neq1$ for some $u\in\Mfrak$ and for every $w\in\Mfrak$, it holds that $|T_a(w)|=\Omc(\lmc(\phi))$. It is clear that given such a~model and a~state~$w$, it takes only polynomial time w.r.t.\ $\lmc(\phi)$ to determine whether $v(\phi,w)=1$. Thus, the validity of single-agent $\modalLinv$-formulas over $\eF$-models is in~$\conp$.

\end{proof}
\section{Proofs of Section~\ref{sec:equivalence}}
In this section, we supply the proofs of Corollary~\ref{cor:clustertree} and Theorem~\ref{theorem:semanticsequivalence}. For the sake of brevity, we assume that $\Amsf=\{a,b\}$. The cases of larger sets of agents can be dealt with in a~similar manner.

We begin with a~technical notion.
\begin{definition}[Modal alternation depth]\label{def:alternationdepth}
Let $\phi\in\modalLinv(\Amsf)$ be represented as a~\emph{syntactical rooted tree} and let further $\Amsf=\{a_1,\ldots,a_m\}$.

Let $\nmbf$ be a~node in the tree for $\phi$. We say that it is \emph{modal} if $\nmbf\in\{\Box_a\mid a\in\Amsf\}$ and \emph{propositional}, otherwise. Let, further, $\pmc(\nmbf)$ be the \emph{parent of $\nmbf$}. We define the \emph{alternation number of~$\nmbf$} ($\Nmc(\nmbf)$) as follows.
\begin{itemize}[noitemsep,topsep=1pt]
\item If $\nmbf$ is the root of the tree and propositional, then $\Nmc(\nmbf)=0$.
\item If $\nmbf$ is the root of the tree and modal, then $\Nmc(\nmbf)=1$.
\item If $\nmbf$ is \emph{not the root} and propositional, then $\Nmc(\nmbf)=\Nmc(\pmc(\nmbf))$.
\item If $\nmbf$ is \emph{not the root}, is modal, and the previous \emph{modal} node $\nmbf$ is s.t.\ $\nmbf'=\nmbf$, then $\Nmc(\nmbf)=\Nmc(\pmc(\nmbf))$.
\item If $\nmbf$ is \emph{not the root}, is modal, and the previous \emph{modal} node $\nmbf$ is s.t.\ $\nmbf'\neq\nmbf$ or all nodes above $\nmbf$ are propositional, then $\Nmc(\nmbf)=\Nmc(\pmc(\nmbf))+1$.
\end{itemize}

The modal alternation depth of $\phi$ ($\Amc(\phi)$) is defined as follows:
\begin{align*}
\Amc(\phi)&=\max\{\Nmc(\nmbf)\mid\nmbf\text{ is a~leaf in the tree of }\phi\}
\end{align*}
\end{definition}

For example, let $\phi=\Box_a(\Box_a(p\wedge\invol\Box_a q)\wedge\invol\Box_b(r\rightarrow\invol\Box_b(s\wedge\Box_at)))\wedge u$. Its syntactical tree is given in Fig.~\ref{fig:syntacticaltree}. To illustrate the definition, we write alternation numbers for each node. One can see that $\Amc(\phi)=3$.
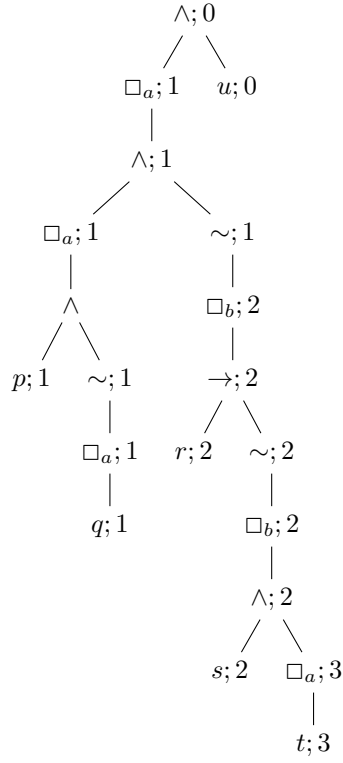
\begin{figure}
\centering
\begin{forest}
[$\wedge;0$[$\Box_a;1$
[$\wedge;1$
[$\Box_a;1$[$\wedge$[$p;1$][$\invol;1$[$\Box_a;1$[$q;1$]]]]][$\invol;1$[$\Box_b;2$[$\rightarrow;2$[$r;2$][$\invol;2$[$\Box_b;2$[$\wedge;2$[$s;2$][$\Box_a;3$[$t;3$]]]]]]]]]][$u;0$]]
\end{forest}
\caption{Syntactical tree for $\Box_a(\Box_a(p\wedge\invol\Box_a q)\wedge\invol\Box_b(r\rightarrow\invol\Box_b(s\wedge\Box_at)))\wedge u$.}
\label{fig:syntacticaltree}
\end{figure}

We now show that $\eF$-satisfiability can be checked over cluster-tree models.
\clustertree*
\begin{proof}
The statement follows immediately from Theorem~\ref{theorem:TS5invGcompleteness} and the fact that satisfiability and non-validity w.r.t.\ $\Fmsf$-models are reducible to one another. Indeed, observe that our tableaux generate counter-models of $\phi$ that are cluster trees of the needed shape and of the height at most $\Amc(\phi)+1$. To obtain a~model satisfying $v(\phi,w)=x$, we simply build a~tableau starting with $\{w{:}\phi\!\leq\!x,~w{:}\phi\!\geq\!x\}\cup\{w\in\cluster_a\mid a\in\Amsf\}$. As $\Amc(\phi)=\Omc(\lmc(\phi))$ for every $\phi\in\modalLinv$, the result follows.
\end{proof}

Let us proceed to the proof of Theorem~\ref{theorem:semanticsequivalence}. We begin with some lemmas. The following statement can be obtained in a~standard manner by induction on $\phi\in\modalLinv$.
\begin{lemma}\label{lemma:preservation}
Fix a function $g:[0,1]\rightarrow[0,1]$ such that $g(0)=0$, $g(1)=1$, $g(1-x)=1-g(x)$, and $x\leq x'$ iff $g(x)\leq g(x')$. For an $\eF$-model $\Mfrak=\langle W,\langle R_a\rangle_{a\in\Amsf},\langle T_a\rangle_{a\in\Amsf},v\rangle$, define a model $\widehat{\Mfrak}$ such that $\widehat{W}=W$, $\langle\widehat{R}_a\rangle_{a\in\Amsf}=\langle R_a\rangle_{a\in\Amsf}$,  $\langle \widehat{T}_a\rangle_{a\in\Amsf}=\langle g[T]_a\rangle_{a\in\Amsf}$, and $\hat{v}(p,w)=g(v(p,w))$ for every $w\in W$  and $p\in\Prop$. Then it holds that $\hat{v}(\phi,w)=g(v(\phi,w))$ for each $\phi\in\modalLinv$ and $w\in W$.
\end{lemma}

Next, we show that formulas can be transformed in such a~way that one occurrence $\Box_a$ is never in the immediate scope of another.
\begin{definition}\label{def:irredundantlynestedformula}
A~formula is called \emph{irredundantly nested} if in every branch of its syntactic tree between every two nodes $\Box_a$, there is a~node $\Box_{a'}$ for some $a'\neq a$.
\end{definition}

Note that $\Amc(\Box_a\phi)=\Amc(\phi)+1$ for every irredundanly nested~$\phi$. Moreover, $\Amc(\phi)=\Dmc(\phi)$ with $\Dmc(\phi)$ being the modal depth of~$\phi$. The converse, however, is not true: $\Dmc(\Box_a\Box_bp\wedge\Box_a\Box_aq)=\Amc(\Box_a\Box_bp\wedge\Box_a\Box_aq)=2$ but $\Box_a\Box_bp\wedge\Box_a\Box_aq$ is not irredundantly nested.
\begin{lemma}\label{lemma:irredundantreduction}
For every $\phi\!\in\!\modalLinv$, there is an irredundantly nested $\phi^\flat$ s.t.\ \mbox{$\phi\!\leftrightarrow\!\phi^\flat$} is $\crispsfiveinvG$-valid and $\eF$-valid.
\end{lemma}
\begin{proof}
It suffices to check that the following equivalences hold in both semantics.
\begin{align}\label{equ:irredundant1}
\Box_a\Box_a\chi&\leftrightarrow\Box_a\chi&\Box_a\invol\Box_a\chi&\leftrightarrow\invol\Box_a\chi\\
\Box_a(\chi\!\rightarrow\!\Box_a\psi)&\leftrightarrow(\invol\Box_a\invol\chi\!\rightarrow\!\Box_a\psi)&\Box_a\invol(\chi\!\rightarrow\!\Box_a\psi)&\leftrightarrow\invol(\Box_a\chi\!\rightarrow\!\Box_a\psi)\nonumber\\
\Box_a(\Box_a\chi\!\rightarrow\!\psi)&\leftrightarrow(\Box_a\chi\!\rightarrow\!\Box_a\psi)&\Box_a\invol(\Box_a\chi\!\rightarrow\!\psi)&\leftrightarrow\invol(\Box_a\chi\!\rightarrow\!\invol\Box_a\invol\psi)\nonumber\\
\Box_a(\chi\!\rightarrow\!\invol\Box_a\psi)&\leftrightarrow(\invol\Box_a\invol\chi\!\rightarrow\!\Box_a\psi)&\Box_a\invol(\chi\!\rightarrow\!\invol\Box_a\psi)&\leftrightarrow\invol(\Box_a\chi\!\rightarrow\!\invol\Box_a\psi)\nonumber\\
\Box_a(\invol\Box_a\chi\!\rightarrow\!\psi)&\leftrightarrow(\invol\Box_a\chi\!\rightarrow\!\Box_a\psi)&\Box_a\invol(\invol\Box_a\chi\!\rightarrow\!\psi)&\leftrightarrow\invol(\invol\Box_a\chi\!\rightarrow\!\invol\Box_a\invol\psi)\nonumber\\
&&\Box_a\invol(\chi\wedge\Box_a\psi)&\leftrightarrow\invol(\Box_a\invol\chi\wedge\invol\Box_a\psi)\nonumber\\
&&\Box_a\invol(\chi\wedge\invol\Box_a\psi)&\leftrightarrow\invol(\Box_a\invol\chi\wedge\Box_a\psi)\nonumber
\end{align}
Note that all these equivalences are \emph{single-agent}. Moreover, we can transform all of them into the negation normal form using the following validities: $\invol\Box_a\phi\leftrightarrow\lozenge_a\invol\phi$ and $\invol(\chi\rightarrow\psi)\leftrightarrow(\invol\psi\coimplies\invol\chi)$. Thus, \eqref{equ:irredundant1} can be replaced with the following equivalences.
\begin{align}\label{equ:irredundant2}
\Box_a\Box_a\chi&\leftrightarrow\Box_a\chi&\Box_a\lozenge_a\chi&\leftrightarrow\lozenge_a\chi\\
\Box_a(\chi\rightarrow\Box_a\psi)&\leftrightarrow(\lozenge_a\chi\rightarrow\Box_a\psi)&\Box_a(\lozenge_a\psi\coimplies\chi)&\leftrightarrow(\lozenge_a\psi\coimplies\lozenge_a\chi)\nonumber\\
\Box_a(\Box_a\chi\rightarrow\psi)&\leftrightarrow(\Box_a\chi\rightarrow\Box_a\psi)&\Box_a(\psi\coimplies\lozenge_a\chi)&\leftrightarrow(\Box_a\psi\coimplies\lozenge_a\chi)\nonumber\\
\Box_a(\chi\rightarrow\lozenge_a\psi)&\leftrightarrow(\lozenge_a\chi\rightarrow\lozenge_a\psi)&\Box_a(\Box_a\psi\coimplies\chi)&\leftrightarrow(\Box_a\psi\coimplies\lozenge_a\chi)\nonumber\\
\Box_a(\lozenge_a\chi\rightarrow\psi)&\leftrightarrow(\lozenge_a\chi\rightarrow\Box_a\psi)&\Box_a(\psi\coimplies\Box_a\chi)&\leftrightarrow(\Box_a\psi\coimplies\Box_a\chi)\nonumber\\
&&\Box_a(\chi\vee\lozenge_a\psi)&\leftrightarrow(\Box_a\chi\vee\lozenge_a\psi)\nonumber\\
&&\Box_a(\chi\vee\Box_a\psi)&\leftrightarrow(\Box_a\chi\vee\Box_a\psi)\nonumber
\end{align}
As one can see, all propositional connectives in~\eqref{equ:irredundant2} are \emph{order-based} (i.e., $\wedge$, $\vee$, $\rightarrow$, or $\coimplies$). From~\cite{CaicedoMetcalfeRodriguezRogger2017}, we know that the single-agent fragment of our semantics is equivalent to the standard semantics. Thus, it suffices to verify the equivalences above in the standard semantics of $\crispsfiveinvG$.

The validity of equivalences in~\eqref{equ:irredundant2} without~$\coimplies$ can be easily established by a~routine check. Let us now consider $\phi=\underbrace{\Box_a(\lozenge_a\psi\coimplies\chi)}_{\phi_L}\leftrightarrow\underbrace{(\lozenge_a\psi\coimplies\lozenge_a\chi)}_{\phi_R}$ as an example. We note briefly that $v(\phi_L,w)\in\{0,v(\Box_a\lozenge_a\psi,w)\}$. Indeed, $v(\lozenge_a\psi\coimplies\chi,w')\in\{0,v(\lozenge_a\psi,w')\}$ for every $R_a$-accessible $w'$, whence, the infimum of such values is either $0$, or is equal to the infimum of the value of $\lozenge_a\psi$, i.e., to the value of $\Box_a\lozenge_a\psi$ at~$w$.

Similarly, $v(\phi_R,w)\in\{0,v(\lozenge_a\psi,w)\}$. Note also that $\Box_a\lozenge_a\psi\leftrightarrow\lozenge_a\psi$ is $\crispsfiveinvG$-valid. Thus, it remains to see that the following two statements hold:
\begin{enumerate}
\item if $v(\Box_a(\lozenge_a\psi\coimplies\chi),w)=0$, then $v(\lozenge_a\psi\coimplies\lozenge_a\chi,w)=0$;
\item if $v(\Box_a(\lozenge_a\psi\coimplies\chi),w)=v(\Box_a\lozenge_a\psi,w)$, then $v(\lozenge_a\psi\coimplies\lozenge_a\chi,w)=v(\lozenge_a\psi,w)$.
\end{enumerate}
For Item~1, we have that
\begin{align*}
v(\Box_a(\lozenge_a\psi\!\coimplies\!\chi),w)\!=\!0&\Rightarrow\left[\begin{matrix}\inf\{v(\lozenge_a\psi,w')\mid wR_aw'\}=0\\\text{ or }\\\exists wR_aw':v(\lozenge_a\psi,w')\leq v(\chi,w')\end{matrix}\right]\\
&\Rightarrow\left[\begin{matrix}\inf\{v(\psi,w')\mid wR_aw'\}=0\\\text{ or }\\\exists wR_aw'\!:\sup\{v(\psi,w'')\!\mid\! w''\!\in\!R_a(w)\}\!\leq\!v(\chi,w')\end{matrix}\right]\\
&\Rightarrow\left[\begin{matrix}v(\lozenge_a\psi,w)=0\\\text{ or }\\\sup\limits_{w''\in R_a(w)}v(\psi,w'')\leq\sup\limits_{wR_aw'}v(\chi,w')\end{matrix}\right]\\
&\Rightarrow v(\lozenge_a\psi\coimplies\lozenge_a\chi,w)=0
\end{align*}

For Item~2, we assume w.l.o.g.\ that $v(\Box_a\lozenge_a\psi,w)=v(\lozenge_a\psi,w)>0$ and proceed as follows:
\begin{align*}
v(\Box_a(\lozenge_a\psi\coimplies\chi),w)\!=\!v(\lozenge_a\psi,w)&\Rightarrow\inf\limits_{wR_aw'}v(\lozenge_a\psi\coimplies\chi,w)=v(\lozenge_a\psi,w)\\
&\Rightarrow\sup\limits_{wR_aw'}v(\psi,w')>\sup\limits_{w''\in R_a(w)}v(\chi,w'')\\
&\Rightarrow v(\lozenge_a\psi,w)>v(\lozenge_a\chi,w)\\
&\Rightarrow v(\lozenge_a\psi\coimplies\lozenge_a\chi,w)=v(\lozenge_a\psi,w)
\end{align*}
The remaining cases can be dealt with in a~similar manner.
\end{proof}

We recall briefly that a~\emph{pointed model} is a~pair $\langle\Mfrak,w\rangle$ with $\Mfrak$ being a~model ($\crispsfiveinvG$-model or $\eF$-model, in our case) and $w$ a~state in~$\Mfrak$. We can now use Lemmas~\ref{lemma:preservation} and~\ref{lemma:irredundantreduction} to show that pointed cluster-tree $\eF$-models $\langle\Mfrak, w\rangle$ can be transformed into pointed $\crispsfiveinvG$-models $\langle\widehat{\Mfrak},\widehat{w}\rangle$ s.t.\ $v(\phi,w)=\widehat{v}(\phi,\widehat{w})$ provided that $\Amc(\phi)$ is bounded by~$\Hmc(\Mfrak)$. The next statement can be obtained by combining Lemmas~2 and~4 from~\cite{CaicedoMetcalfeRodriguezRogger2013}.
\begin{lemma}\label{lemma:eFmodeltomodel}
For every pointed cluster tree $\eF$-model~$\langle\Mfrak,w_0\rangle$ with $w_0$ belonging to the root cluster $\cluster_a$ of~$\Mfrak$ and every \mbox{$\phi\!\in\!\modalLinv$} s.t.\ $\Amc(\phi)\!\leq\!\Hmc(\Mfrak)$, there is a~pointed $\crispsfiveinvG$-model $\langle\widehat{\Mfrak},\widehat{w}_0\rangle$ s.t.\ \mbox{$v(\phi,w_0)=\widehat{v}(\phi,\widehat{w}_0)$.}
\end{lemma}
\begin{proof}
Note that we can w.l.o.g.\ assume that $\Mfrak=\langle W,\langle R_a\rangle_{a\in\Amsf},\langle T_a\rangle_{a\in\Amsf},v\rangle$ is finite, that $\cluster_a=\{w_0,\ldots,w_m\}$, and that $\phi$ is irredundantly nested. For simplicity, we will assume here that $\Amsf=\{a,b\}$. The proofs for larger sets of agents can be obtained in a~similar fashion.

We proceed by double induction on $\Hmc(\Mfrak)$ and on $\phi\in\modalLinv$. First, consider the basis case of $\Hmc(\Mfrak)=1$. That is, $\Mfrak$ contains one proper cluster, namely,~$\cluster_a$. Moreover, $\Amc(\phi)\leq1$. If $\Amc(\phi)=0$, then $\phi$ is $\Box$-free (propositional), and the statement follows immediately since the semantical conditions for propositional formulas in $\eF$-models and $\crispsfiveinvG$-models coincide. Now let $\Amc(\phi)=1$.

Note that since $T_a(w)=T_a(w')$ for every $w$ and $w'$ s.t.\ $wR_aw$, we can assume that $T_a$'s and $T_b$'s are associated with clusters and not states. We let $T_a(\cluster_a)=\{\alpha_1,\ldots,\alpha_n\}$ and $T_b(w_0)=\{\beta_1,\ldots,\beta_r\}$\footnote{Recall that $\Mfrak$ contains only one proper cluster.} and construct a~pointed $\crispsfiveinvG$-model $\langle\widehat{\Mfrak},\widehat{w}_0\rangle$ s.t.\ $\widehat{v}(\phi,\widehat{w}_0)=v(\phi,w_0)$. It suffices to consider the following cases:
\begin{enumerate}[noitemsep,topsep=1pt]
\item $\phi=\Box_a\chi$ for some $\chi\in\modalLinv$ s.t.\ $\chi$ is propositional;
\item $\phi=\Box_b\chi$ for some $\chi\in\modalLinv$ s.t.\ $\chi$ is propositional;
\item $\phi=\chi\circ\psi$ s.t.\ w.l.o.g.\ $\chi$ contains $\Box_a$ and $\psi$ contains $\Box_b$ with $\circ\in\{\wedge,\rightarrow\}$.
\end{enumerate}

For Item~1, note that $v(\Box_a\chi,w_0)\!\leq\!\inf\{v(\chi,w')\!\mid w_0R_aw'\}$. Moreover, as $\Mfrak$~is finite, $v(\Box_a\chi,w_0)=\inf\{v(\chi,w')\mid w_0R_aw'\}$ iff there are some $\alpha_i\in T(\cluster_a)$ and $w''\in\cluster_a$ such that $\inf\{v(\chi,w')\mid w_0R_aw'\}=\alpha_i=v(\chi,w'')$. In this case, we can set $\widehat{\Mfrak}=\langle W,\langle R_a\rangle_{a\in\Amsf},v\rangle$ (i.e., just remove $T_a$'s from~$\Mfrak$) because the $\Fmsf$-models and $\crispsfiveinvG$-models evaluate propositional formulas in the same way. This also deals with the case of $v(\Box\psi,w_0)=1$.

Now let $v(\Box\psi,w_{0})=\alpha_i\neq 1$, and, furthermore, $\alpha_i<\inf\{v(\chi,w')\mid w_0R_aw'\}$. We define order embeddings $h^a_k$'s for every $k<\omega$. We let $h^a_0(x)=x$ for every $x\in[0,1]$ and for the remaining ones, we demand that $h^a_{k}(0)=0$, $h^a_{k}(1)=1$, and $h^a_k(1-j)=1-h^a_k(j)$ for all $j\in[0,1]$.
\begin{align*}
h^a_{k}(\alpha_{i})&=\alpha_{i} \mbox{ for all }i\leq n\mbox{ and }k\in\Nmbb\\
h^a_{k}[(\alpha_{i},\alpha_{i+1})]&=(\alpha_{i},\min(\alpha_{i}+\sfrac{1}{k},\alpha_{i+1}))\mbox{ for all }i\leq n-1\mbox{ and even }k\in\Nmbb\\
h^a_{k}[(\alpha_{i},\alpha_{i+1})]&=(\max(\alpha_{i}-\sfrac{1}{k},\alpha_{i+1}),\alpha_{i+1})\mbox{ for all }i\leq n-1\mbox{ and odd }k\in\Nmbb
\end{align*}
We then copy $\cluster_a$ $\omega$ times. I.e., $\widehat{\Mfrak}=\langle\widehat{W},\widehat{R}_a,\widehat{R}_b,\widehat{v}\rangle$ is as follows:
\begin{align*}
\widehat{W}&=\bigcup\limits_{k<\omega}\{w^k_i\mid i\leq m\}\\
\widehat{R}_a&=\widehat{W}\times\widehat{W}\\
\widehat{R}_b&=\{\langle\widehat{w},\widehat{w}\rangle\mid\widehat{w}\in\widehat{W}\}\\
\widehat{v}(p,w^k_i)&=h^a_k(v(p,w_i))
\end{align*}

We have $v(\psi,w_j)\geq\alpha_i$ for all $w_j\in\cluster_a$, and by the construction of the order embeddings, we get that for all $w'\in\widehat{W}$, $\hat{v}(\psi,w')\geq\alpha_i$. Now, there must be some $w''\in\cluster_a$ for which $v(\psi,w'')\in (\alpha_i,\alpha_{i+1})$. By construction of the order embeddings, for any $\varepsilon>0$, there is some $k\in\Nmbb$ for which $h_k(v(\psi,w''))=\hat{v}(\psi,\hat{w}''^k)\in(\alpha_i,\alpha_i+\varepsilon)$. Hence, $\hat{v}(\Box\psi,\hat{w}^0_0)=\alpha_i=v(\Box\psi,w_0)$.

We note briefly that if $\Amc(\phi)=1$ and $\phi$ contains several subformulas $\Box_a\chi_1$, \ldots, $\Box_a\chi_q$ (but does not contain $\Box_b$'s), the proof proceeds in the same way. The only difference is that we will need to copy $\cluster_a$ if the value of \emph{at least one} $\Box_a\chi_i$ is lower than the infimum of values of~$\chi_i$ in~$\cluster_a$. Still, one can see from the construction that $h^a_k$'s preserve the value of \emph{all} formulas of the form $\Box_a\chi_i$ in $w_0$, whence, it suffices to produce $\omega$ copies of~$\cluster_a$ only once.

The case of Item~2 is similar but we need different model and embeddings. Namely, we set $h^b_0(x)=x$, and for $k>0$, $h^b_k$'s are as follows:
\begin{align*}
h^b_{k}(\beta_{i})&=\beta_{i}\mbox{ for all }i\leq r\mbox{ and }k\in\Nmbb\\
h^b_{k}[(\beta_{i},\beta_{i+1})]&=(\beta_{i},\min(\beta_{i}+\sfrac{1}{k},\beta_{i+1})) \mbox{ for all }i\leq r-1\mbox{ and even }k\in\Nmbb\\
h^b_{k}[(\beta_{i},\beta_{i+1})]&=(\max(\beta_{i}-\sfrac{1}{k},\beta_{i+1}),\beta_{i+1}) \mbox{ for all }i\leq r-1\mbox{ and odd }k\in\Nmbb
\end{align*}
We also define $\widehat{\Mfrak}=\langle\widehat{W},\widehat{R}_a,\widehat{R}_b,\widehat{v}\rangle$ as follows:
\begin{align*}
\widehat{W}&=\{\widehat{w}_i\mid i\leq m\}\cup\{\widehat{w}^k_0\mid 1\leq k<\omega\}\\
\widehat{R}_a&=\{\widehat{w}_i\mid i\leq m\}\times\{\widehat{w}_i\mid i\leq m\}\cup\{\langle\widehat{w}^k_0,\widehat{w}^k_0\rangle\mid 1\leq k<\omega\}\\
\widehat{R}_b&=\{\langle s,s'\rangle\mid\{s,s'\}\subseteq\{\widehat{w}_0\}\cup\{\widehat{w}^k_0\mid 1\leq k<\omega\}\}\\
\widehat{v}(p,\widehat{w}_i)&=v(p,w_i)\\
\widehat{v}(p,\widehat{w}^k_0)&=h^b_k(v(p,w_0))
\end{align*}
Here, $h^b_k$'s are essentially the same as $h^a_k$'s but defined w.r.t.\ $T_b(w_0)$, not $T_a(\cluster_a)$. Namely, 
Note, moreover, that since $\Hmc(\Mfrak)=1$, $R_b(w)=\{w\}$ for all $w\in\cluster_a$. The rest of the proof that $v(\phi,w_0)=\widehat{v}(\phi,\widehat{w}_0)$ is the same as in the previous case.

Finally, in the third case, we combine the two constructions by defining $\widehat{\Mfrak}\!=\!\langle\widehat{W},\widehat{R}_a,\widehat{R}_b,\widehat{v}\rangle$ is as follows:
\begin{align*}
\widehat{W}&=\underbrace{\bigcup\limits_{k<\omega}\{w^k_i\mid i\leq m\}}_{W_{R_a}}\cup\underbrace{\{\widehat{w}^k_0\mid 1\leq k<\omega\}}_{W_{R_b}}\\
\widehat{R}_a&=\bigcup\limits_{k<\omega}\{w^k_i\mid i\leq m\}\times\bigcup\limits_{k<\omega}\{w^k_i\mid i\leq m\}\cup\{\langle\widehat{w}^k_0,\widehat{w}^k_0\rangle\mid 1\leq k<\omega\}\\
\widehat{R}_b&=\{\langle s,s'\rangle\mid\{s,s'\}\subseteq\{w^0_0\}\cup\{\widehat{w}^k_0\mid 1\!<\!k\!<\!\omega\}\}\cup\!\!\bigcup\limits_{1\leq k<\omega}\!\!\!\!\{\langle w^k_i,w^k_i\rangle\mid i\!\leq\!m\}\\
\widehat{v}(p,w^k_i)&=h^a_k(v(p,w_i))\\
\widehat{v}(p,\widehat{w}^k_0)&=h^b_k(v(p,w_0))
\end{align*}
Again, the proof that $v(\phi,w_0)=\widehat{v}(\phi,w^0_0)$ is essentially the same as that in the previous cases but now we also use the fact that the value of $\chi$ depends only on~$W_{R_a}$ and the value of $\psi$ only on~$W_{R_b}$.\footnote{Note that even for formulas such as $\phi=\Box_a\chi_1\rightarrow(\Box_b\psi_1\rightarrow((\Box_a\chi_2\wedge\tau_1)\rightarrow(\tau_2\wedge\Box_b\psi_2)))$ with $\Amc(\phi)=1$, the values of all its $\Box_a$-subformulas still depend only on~$W_{R_a}$, the values of $\Box_b$-subformulas only on~$W_{R_b}$, and the values of propositional subformulas only on $w^0_0$. As the values of modal subformulas are preserved using $h^a_k$'s and $h^b_k$'s, it suffices to consider the case when $\phi=\chi\circ\psi$ with $\chi$ and $\psi$ containing $\Box_a$ and $\Box_b$, respectively.}

Let us now proceed to the induction step. Assume that $\Hmc(\Mfrak)=n+1$ and the statement holds for all cluster-tree $\eF$-models of the height $n$ or lower. Again, we conduct induction on~$\phi$. The basis case of $\phi=p$ holds trivially, the cases of $\phi=\invol\chi$, $\phi=\chi\wedge\psi$, and $\phi=\chi\rightarrow\psi$ can be obtained by straightforward applications of the induction hypothesis. It remains to consider modal cases.

Let $\phi=\Box_a\chi$ and let further, $v(\Box_a\chi,w_0)=x$. We now have two cases: (I)~$\inf\{v(\chi,w')\mid wR_aw'\}=x$ or (II) $\inf\{v(\chi,w')\mid wR_aw'\}>x$. In the case~(I), we replace every $w_i\in\cluster_a$ with a~pointed $\crispsfiveinvG$-model $\langle\Mfrak_{w_i},\widehat{w}^i_0\rangle$ s.t.\ $v^{\Mfrak_{w_i}}(\chi,\widehat{w}^i_0)=v(\chi,w_i)$. Note that since $\Box_a\chi$ is irredundantly nested, we have $\Amc(\chi)=\Amc(\Box_a\chi)-1$. Thus, by Corollary~\ref{cor:clustertree}, we can apply the induction hypothesis to obtain the needed $\langle\Mfrak_{w_i},\widehat{w}^i_0\rangle$'s. We now extend accessibility relations on $\widehat{w}^i_0$'s as expected: $\widehat{w}^i_0R_a\widehat{w}^j_0$ and $\widehat{w}^i_0R_b\widehat{w}^j_0$ for every $i,j\in\{0,\ldots,m\}$. Clearly, $v(\Box_a\chi,w_0)=v^{\Mfrak_{w_0}}(\Box_a\chi,\widehat{w}^0_0)$.

In the case~(II), we replace each $w_i\in\cluster_a$ with countably many pointed $\crispsfiveinvG$-models $\langle\Mfrak^k_{w_i},\widehat{w}^{k,i}_0\rangle$ s.t.\ $v^{\Mfrak^k_{w_i}}(\chi,\widehat{w}^{k,i}_0)=h^a_k(v(\chi,w_i))$. Again, this is possible to have by the induction hypothesis because $\Amc(\chi)=\Amc(\Box_a\chi)-1$. One can now show that $v(\Box_a\chi,w_0)=v^{\Mfrak^k_{w_0}}(\Box_a\chi,\widehat{w}^{0,0}_0)$ utilising the argument in the basis case.

Finally, let $\phi=\Box_b\chi$. The procedure is similar to the case of $\phi=\Box_a\chi$ but we will need to replace the states in $R_b(w_0)$ and use $h^b_k$'s. The result follows.
\end{proof}

\semanticsequivalence*
\begin{proof}
The left-to-right direction follows from Corollary~\ref{cor:clustertree} and Lemma~\ref{lemma:eFmodeltomodel}. For right-to-left, suppose that $\phi$ is not $\crispsfiveinvG$-valid. Then there exists a~$\crispsfiveinvG$-countermodel $\Mfrak=\langle W,\langle R_{a}\rangle_{a\in\Amsf},v\rangle$ with $w\in W$ s.t.\ $v(\phi,w)=x$ for some $x<1$.

For every $a\in\Amsf$, we set
\begin{align*}
[\phi]_{\Box_a}^{\sim}&=\lbrace \Box_a\psi\mid\Box_a\psi\mbox{ is a subformula of }\phi\rbrace\!\cup\!\lbrace\invol\Box_a\psi\mid\Box_a\psi\mbox{ is a subformula of }\phi\rbrace
\end{align*}
and define a function $\widehat{T}_a$ for each $s\in W$ with the following scheme:
\begin{align*}
\widehat{T}_a(s)&=\lbrace\inf\lbrace v(\psi,w')\mid sR_aw'\rbrace\mid \psi\in[\phi]_{\Box_a}^{\sim}\rbrace\cup\lbrace 0,\tfrac{1}{2},1\rbrace
\end{align*}%
Let $\widehat{\Mfrak}=\langle \widehat{W},\langle \widehat{R}_a\rangle_{a\in\Amsf},\langle \widehat{T}_a\rangle_{a\in\Amsf},\hat{v}\rangle$ where $\widehat{W}=W$, $\widehat{R_a}=R_a$ for all $a\in\Amsf$, and $\hat{v}=v$. By definition, $\langle\widehat{W},\langle\widehat{R}_a\rangle_{a\in\Amsf}\rangle$ is a frame and each $\widehat{T}_a$ and point $s$, $\widehat{T}_a(s)$ is finite and closed under $1-x$, whence $\widehat{\Mfrak}$ is an $\Fmsf$-model. %
We claim that for all $w$ and all $\psi$ a subformula of $\phi$, $\hat{v}(\psi,w)=v(\psi,w)$.

The basis step in case $\phi$ is atomic is immediate and the propositional connectives follow easily, leaving the case of an operator $\Box_a$. Let $\psi=\Box_a\xi$. Then $v(\Box_a\xi,w)=\inf\lbrace v(\xi,w')\mid wR_aw'\rbrace$, whence by construction, $v(\Box_a\xi,w)$ is an element $x$ in $\widehat{T}_a(w)$ such that for all $w'$ such that $w\widehat{R}_aw'$, $v(\xi,w')\leq x$; as it is the greatest such element, $\hat{v}(\Box_a\xi,w)=x$, whence $\hat{v}(\Box_a\xi,w)=v(\Box_a\xi,w)$. The case of $\psi=\Box_b\xi$ can be considered in a~similar manner. This completes the induction, so $\widehat{\Mfrak}$ provides an $\Fmsf$-countermodel.
\end{proof}
\end{document}